\documentclass[12pt]{amsart}       
\usepackage{txfonts}
\usepackage{amssymb}
\usepackage{eucal}
\usepackage{bbm}
\usepackage{tikz}
\usetikzlibrary{calc} %
\usetikzlibrary{positioning} %
\usepackage{graphicx}
\usepackage{amsmath}
\usepackage{amscd}
\usepackage{algorithmicx}
\usepackage{float}
\usepackage{caption}
\usepackage{enumitem}
\usepackage{algpseudocode}
\usepackage{algorithm}
\usepackage[all]{xy}           
\usepackage{amsfonts,latexsym}
\usepackage{xspace,amsgen}
\usepackage{epstopdf}
\usepackage{array}
\usepackage{float}
\usepackage[toc,page]{appendix}
\usepackage{bookmark}
\usepackage{color}
\allowdisplaybreaks[4]
\numberwithin{equation}{section}
\usepackage{fancybox}
\usepackage{colordvi}
\usepackage{multicol}
\newcommand{\Par}{\mathrm{Par}}
\newcommand{\curvearrowrightl}{\curvearrowright_{l}}
\usepackage[active]{srcltx} 
\usepackage{graphicx}

\usepackage[left=2.5cm,right=2.5cm,top=2.5cm,bottom=2.5cm]{geometry}

\newtheorem{theorem}{Theorem}[section]
\newtheorem{proposition}[theorem]{Proposition}    
\newtheorem{lemma}[theorem]{Lemma}
\newtheorem{coro}[theorem]{Corollary}
\newtheorem{prop-def}[theorem]{Proposition-Definition}  
\newtheorem{coro-def}[theorem]{Corollary-Definition}
\newtheorem{prop}[theorem]{ Proposition}  

\theoremstyle{definition}
\newtheorem{definition}[theorem]{Definition}
\newtheorem{defn}[theorem]{Definition}
\newtheorem{remark}[theorem]{Remark}

\newtheorem{exam}[theorem]{Example}

\renewcommand{\thefootnote}{}

\newcommand{\wt}{\text{wt}}

\setlist[enumerate,1]{label={\rm(\alph*)}}
\setlist[enumerate,2]{label={\rm(\roman*)}}

\newcommand{\nc}{\newcommand}

\newcommand{\TreeDotSize}{4.5pt}
\newcommand{\TreeXSpread}{0.45}
\newcommand{\TreeYStep}{0.60}
\newcommand{\TreeLabelSep}{0.2pt}
\newcommand{\NAcell}{\multicolumn{1}{c|}{\textemdash}}
\newcommand{\bij}{\hookrightarrow\kern-1.2ex\twoheadrightarrow}

\tikzset{
  dot/.style={
    circle,
    fill,
    inner sep=0pt,
    minimum size=\TreeDotSize
  },
  every node/.style={font=\small}
}

\newcommand{\TIKZ}[1]{%
  \vcenter{\hbox{%
    \tikz[baseline={(current bounding box.center)}]{#1}%
  }}%
}

\newcommand{\Tsingle}[1]{%
  \TIKZ{%
    \node[dot] (x) at (0,0) {};
    \node[right=1pt] at (x) {$#1$};
  }%
}

\newcommand{\Tbcad}{%
  \TIKZ{%
    \node[dot] (a) at (0,0) {}; \node[right=1pt] at (a) {$a$};
    \node[dot] (b) at (-\TreeXSpread,\TreeYStep) {}; \node[right=1pt] at (b) {$b$};
    \node[dot] (c) at ( \TreeXSpread,\TreeYStep) {}; \node[right=1pt] at (c) {$c$};
    \node[dot] (d) at ( \TreeXSpread,2*\TreeYStep) {}; \node[right=1pt] at (d) {$d$};
    \draw (a)--(b) (a)--(c) (c)--(d);
  }%
}

\newcommand{\Tacd}{%
  \TIKZ{%
    \node[dot] (a) at (0,0) {}; \node[right=1pt] at (a) {$a$};
    \node[dot] (c) at (0,\TreeYStep) {}; \node[right=1pt] at (c) {$c$};
    \node[dot] (d) at (0,2*\TreeYStep) {}; \node[right=1pt] at (d) {$d$};
    \draw (a)--(c)--(d);
  }%
}

\newcommand{\Tabc}{%
  \TIKZ{%
    \node[dot] (a) at (0,0) {}; \node[right=1pt] at (a) {$a$};
    \node[dot] (b) at (-\TreeXSpread,\TreeYStep) {}; \node[right=1pt] at (b) {$b$};
    \node[dot] (c) at ( \TreeXSpread,\TreeYStep) {}; \node[right=1pt] at (c) {$c$};
    \draw (a)--(b) (a)--(c);
  }%
}

\newcommand{\Tac}{%
  \TIKZ{%
    \node[dot] (a) at (0,0) {}; \node[right=1pt] at (a) {$a$};
    \node[dot] (c) at (0,1.25*\TreeYStep) {}; \node[right=1pt] at (c) {$c$};
    \draw (a)--(c);
  }%
}

\newcommand{\Tcd}{%
  \TIKZ{%
    \node[dot] (c) at (0,0) {}; \node[right=1pt] at (c) {$c$};
    \node[dot] (d) at (0,1.25*\TreeYStep) {}; \node[right=1pt] at (d) {$d$};
    \draw (c)--(d);
  }%
}

\nc{\tred}[1]{\textcolor{red}{#1}}
\nc{\tblue}[1]{\textcolor{blue}{#1}}
\nc{\tgreen}[1]{\textcolor{green}{#1}}
\nc{\tpurple}[1]{\textcolor{purple}{#1}}
\nc{\btred}[1]{\textcolor{red}{\bf #1}}
\nc{\btblue}[1]{\textcolor{blue}{\bf #1}}
\nc{\btgreen}[1]{\textcolor{green}{\bf #1}}
\nc{\btpurple}[1]{\textcolor{purple}{\bf #1}}
\nc{\NN}{{\mathbb N}}
\nc{\la}{L^+_a} \nc{\lam}{L^{-}_{a}}
\nc{\ncsha}{{\mbox{\cyr X}^{\mathrm NC}}} \nc{\ncshao}{{\mbox{\cyr
X}^{\mathrm NC}_0}}

\newcommand{\inj}{\lhook\joinrel\longrightarrow}
\renewcommand{\bij}{\widetilde{\longrightarrow}}
\newcommand{\shuffle}{{\sqcup\hskip -1.7pt\sqcup}}

\newcommand{\delete}[1]{}

\nc{\mlabel}[1]{\label{#1}}
\nc{\mcite}[1]{\cite{#1}}
\nc{\mref}[1]{\ref{#1}}
\nc{\meqref}[1]{\eqref{#1}}
\nc{\mbibitem}[1]{\bibitem{#1}}

\delete{
\nc{\mlabel}[1]{\label{#1}{\hfill \hspace{1cm}{\bf{{\ }\hfill(#1)}}}}
\nc{\mcite}[1]{\cite{#1}{{\bf{{\ }(#1)}}}}
\nc{\mref}[1]{\ref{#1}{{\bf{{\ }(#1)}}}}
\nc{\meqref}[1]{\eqref{#1}{{\bf{{\ }(#1)}}}}
\nc{\mbibitem}[1]{\bibitem[\bf #1]{#1}}
}
\newfont{\scyr}{wncyr10 scaled 650}
\nc{\sha}{{\mbox{\scyr X}}}  
\nc{\ssha}{\mbox{\bf \scyr X}}
\nc{\shap}{{\mbox{\cyrs X}}} 
\nc{\shpr}{\diamond}    
\nc{\shp}{\ast} \nc{\shplus}{\shpr^+}
\nc{\shprc}{\shpr_c}    
\nc{\dep}{\mrm{dep}} \nc{\lc}{\lfloor} \nc{\rc}{\rfloor}
\nc{\db}{\leq_{\rm db}} \nc{\bfk}{{\bf k}}

\nc{\cala}{{\mathcal A}} \nc{\calb}{{\mathcal B}}
\nc{\calc}{{\mathcal C}}
\nc{\cald}{{\mathcal D}} \nc{\cale}{{\mathcal E}}
\nc{\calf}{{\mathcal F}} \nc{\calg}{{\mathcal G}}
\nc{\calh}{{\mathcal H}} \nc{\cali}{{\mathcal I}}
\nc{\call}{{\mathcal L}} \nc{\calm}{{\mathcal M}}
\nc{\caln}{{\mathcal N}} \nc{\calo}{{\mathcal O}}
\nc{\calp}{{\mathcal P}} \nc{\calr}{{\mathcal R}}
\nc{\cals}{{\mathcal S}} \nc{\calt}{{\mathcal T}}
\nc{\calu}{{\mathcal U}} \nc{\calw}{{\mathcal W}} \nc{\calk}{{\mathcal K}}
\nc{\calx}{{\mathcal X}} \nc{\CA}{\mathcal{A}}

\nc{\fraka}{{\mathfrak a}} \nc{\frakA}{{\mathfrak A}}
\nc{\frakb}{{\mathfrak b}} \nc{\frakB}{{\mathfrak B}}
\nc{\frakc}{{\mathfrak c}}
\nc{\frakD}{{\mathfrak D}} \nc{\frakF}{\mathfrak{F}}
\nc{\frakf}{{\mathfrak f}} \nc{\frakg}{{\mathfrak g}}
\nc{\frakH}{{\mathfrak H}} \nc{\frakL}{{\mathfrak L}}
\nc{\frakM}{{\mathfrak M}} \nc{\bfrakM}{\overline{\frakM}}
\nc{\frakm}{{\mathfrak m}} \nc{\frakP}{{\mathfrak P}}
\nc{\frakN}{{\mathfrak N}} \nc{\frakp}{{\mathfrak p}}
\nc{\frakS}{{\mathfrak S}} \nc{\frakT}{\mathfrak{T}}
\nc{\frakX}{{\mathfrak X}}

\font\cyr=wncyr10 \font\cyrs=wncyr7
\nc{\dom}[1]{\textcolor{purple}{#1}}
\nc{\domc}[1]{\textcolor{purple}{$D$:#1}}
\nc{\lir}[1]{\textcolor{red}{Li:#1}}
\nc{\xing}[1]{\textcolor{blue}{Xing:#1}}
\nc{\revise}[1]{\textcolor{blue}{#1}}

\nc{\V}{V} \nc{\pro}{\otimes}
\nc{\ZZ}{\mathbb{Z}} 
\nc{\etree}{\mathbbm{1}}
\nc{\xx}{\mathcal{X}}
\nc{\RP}{{\mathcal{D}}^{\alpha}([0, T]^2, V)}
\nc{\Y}{{\bf Y}}\nc{\id}{\text{id}} 
\nc{\Id}{\text{Id}}\nc{\Z}{{\bf Z}}
\nc{\nca}{{\rm NC(A)}}

\nc{\pna}{{\rm {PN}(A)}}
\renewcommand{\nca}{\pna}
\nc{\bnca}{\overline{{\rm PN}}(A)} 

\nc{\btri}{\blacktriangleright_{w}}
\nc{\lbar}{\overline}
\nc{\da}{a}
\nc{\zia}{z_{i}^{\da}}
\nc{\z}[2]{z_{#1}^{#2}}
\nc{\lp}{[-,-]_{\cdot}}
\nc{\bzia}{\mathbf{z}_{\mathbf{i}}^{\mathbf{a}}}
\nc{\bzjb}{\mathbf{z}_{\mathbf{j}}^{\mathbf{b}}}
\nc{\bzkc}{\mathbf{z}_{\mathbf{k}}^{\mathbf{c}}}
\nc{\bg}{\mathbf{g}}
\nc{\OF}{\mathcal{OF}}
\nc{\OT}{\mathcal{OT}}
\nc{\fer}{\Phi}
\nc{\vz}{\mathcal{Z}}
\nc{\bx}{\mathbf{x}}
\nc{\by}{\mathbf{y}}
\nc{\bz}{\mathbf{z}}
\nc{\bu}{\mathbf{u}}
\nc{\bv}{\mathbf{v}}
\nc{\bw}{\mathbf{w}}
\nc{\bp}{\mathbf{p}}
\nc{\bq}{\mathbf{q}}
\nc{\bzai}[1]{\mathbf{z}_{\mathbf{i}}^{\mathbf{a}^{#1}}}
\nc{\plie}{\rm{PostLie}(A)}
\nc{\xhc}{\mathfrak{c}}
\nc{\dhc}{\mathfrak{C}}
\newcommand{\lppdot}{\mathchoice
  {\scalebox{6}{$\cdot$}}
  {\scalebox{6}{$\cdot$}}
  {\raisebox{-0.80ex}{\scalebox{1.8}{$\cdot$}}}
  {\raisebox{-0.60ex}{\scalebox{1.4}{$\cdot$}}}
}

\newcommand{\lpp}[2]{[#1,#2]_{\lppdot}}
\nc{\Yong}[1]{\textcolor{red}{Yong:#1}} \nc{\FF}{F}

\begin{document}

\title[The planar Hopf algebra of noncommutative multi-indices]{The planar Hopf algebra of noncommutative multi-indices}
%
%
\author{Yong Yu}
\address{School of Mathematics and Statistics, Lanzhou University
Lanzhou, 730000, China
}
\email{yuyong\_lzu@yeah.net}

\author{Xing Gao$^{*}$}\thanks{*Corresponding author}
\address{School of Mathematics and Statistics, Lanzhou University
Lanzhou, 730000, China; Gansu Provincial Research Center for Basic Disciplines of Mathematics
and Statistics, Lanzhou, 730070, China
}
\email{gaoxing@lzu.edu.cn}

\author{Dominique Manchon$^{*}$}
\address{Laboratoire de Math\'{e}matiques Blaise Pascal, CNRS-Universit\'{e} Clermont-Auvergne, 3 place Vasar\'ely, CS 60026, 63178 Aubi\`ere, France}
\email{Dominique.Manchon@uca.fr}
\date{}
\begin{abstract}
We construct the planar Linares--Otto--Tempelmayr Hopf algebra, thereby
filling the missing planar noncommutative multi-index corner in the square
relating the LOT, Butcher--Connes--Kreimer, and Munthe-Kaas--Wright Hopf
algebras. Starting from the free associative algebra on a weighted alphabet $\mathbb Z_{\ge -1}\times A$,
we define an insertion-type product yielding a post-Lie structure on the Lie
algebra generated by the linear span $V(A)$ of weight $-1$ monomials whose proper left prefixes all have nonnegative weight, and the Guin--Oudom
construction then produces the planar LOT Hopf algebra. We introduce a 
planar tree fertility map from decorated planar rooted trees to 
monomials in $V(A)$, prove that it is a linear isomorphism, and obtain a natural Hopf algebra isomorphism with the Munthe-Kaas--Wright Hopf
algebra. We further derive an
explicit coproduct formula in terms of left-admissible cuts, establish the
extraction-contraction coproduct, and construct a word symmetrization
operator compatible with the classical tree symmetrization operator.
\end{abstract}

\makeatletter
\@namedef{subjclassname@2020}{\textup{2020} Mathematics Subject Classification}
\makeatother
\subjclass[2020]{
05C05, 	
17B01, 
16T30, 
17A30. 
}

\keywords{Planar LOT Hopf algebra, post-Lie algebra, Lie algebra, Guin--Oudom construction.}

\begin{titlepage}
\thispagestyle{empty}

\end{titlepage}

\setcounter{page}{1}

\maketitle
\setcounter{footnote}{0}
\renewcommand{\thefootnote}{\arabic{footnote}}
\tableofcontents

\setcounter{section}{0}

\allowdisplaybreaks

\section{Introduction}\label{sec:introduction}

The theory of rough paths, initiated by T. Lyons, provides a robust algebraic and analytic
framework for differential equations driven by highly irregular signals \mcite{Bellingeri2023, lyons1998}.
Over the years it has developed into a rich meeting point of analysis, geometry, and combinatorics
\mcite{HBI2019,MD2025}. On the combinatorial side, rooted trees and rooted forests encode
branched expansions and play a fundamental role in rough differential equations, stochastic analysis,
and related problems \mcite{Manchon2020, Foissy2001, XING2024}. At the same time, multi-index
structures have emerged as an efficient alternative language in regularity structures, rough analysis,
and singular stochastic problems \mcite{YM2019, Bruned2024, 2026BHPZ, Otto2023}.

Besides the classical nonplanar branched setting, planar structures also arise naturally in Lie-Butcher
theory, Lie group integrators, and planarly branched rough paths \mcite{Manchon2020, XING2024, MKW2008, MKL2013}.
This immediately suggests a natural question: is there a planar and noncommutative analogue of the
Hopf algebra of decorated multi-indices, playing for planar rooted forests the same role that the
commutative Linares-Otto-Tempelmayr Hopf algebra plays for ordinary rooted forests? The main
motivation of the present paper is to answer this question positively.

\subsection{BCK and LOT Hopf algebras}

The Butcher-Connes-Kreimer Hopf algebra $\mathcal H_{\mathrm{BCK}}^{A}$ is the basic combinatorial
Hopf algebra on $A$-decorated rooted forests, with coproduct governed by admissible cuts
\mcite{Foissy2001}. Decorated rooted trees realize the free pre-Lie algebra on the set of decorations,
and pre-Lie algebras provide the natural algebraic framework for the rooted-tree side
\mcite{CL2001,DMA2011}. The corresponding Grossman-Larson type structures arise from the
Guin-Oudom construction on enveloping algebras of pre-Lie algebras \mcite{MML2006,JMD2008};
see also \mcite{MK2014,Rahm2021} for related developments.\\

On the commutative multi-index side, the Linares-Otto-Tempelmayr Hopf algebra
$\mathcal H_{\mathrm{LOT}}^{A}$ is rooted in the free Novikov algebra generated by the set of
decorations \mcite{BCZ2017, DA2002, DS2025, SK2023}. The key link with the rooted-tree side is
provided by the fertility map from decorated rooted trees to decorated multi-indices
\mcite{LINARES2023,ZHU2024}. By transposition and multiplicative extension, this map induces a canonical
embedding
\[
j:\mathcal H_{\mathrm{LOT}}^{A}\hookrightarrow \mathcal H_{\mathrm{BCK}}^{A},
\]
thereby relating the commutative multi-index side to the nonplanar rooted-tree side
\mcite{ZHU2024}. A decisive recent advance was achieved in \mcite{ZHU2024}, where the Hopf
structure of $\mathcal H_{\mathrm{LOT}}^{A}$ was made completely explicit through a combinatorial
formula for its coproduct in terms of admissible cuts of decorated multi-indices. The corresponding
extraction-contraction coproduct fits naturally into the general framework of interacting Hopf algebras
\mcite{Calaque2011}.

\subsection{MKW Hopf algebra and the missing corner}

On the planar tree side, the relevant Hopf algebra is the Munthe-Kaas-Wright Hopf algebra
$\mathcal H_{\mathrm{MKW}}^{A}$ of decorated planar rooted forests \mcite{Rahm2023, MKW2008}.
Its coproduct is described by left-admissible cuts, and its graded dual is governed by the ordered
Grossman-Larson product \mcite{MKW2008}. Moreover, decorated planar rooted trees carry the
classical left grafting operation, which yields a post-Lie algebra structure and places the planar
rooted-tree side in the framework of Lie-Butcher theory \mcite{MJH2023, FK2014, LMK2013, MKL2013, Rahm2022}.\\

At this stage a natural square already emerges. The upper horizontal map $j$ is known on the
commutative/nonplanar side, and the right vertical map
\[
\Omega_t:\mathcal H_{\mathrm{BCK}}^{A}\longrightarrow \mathcal H_{\mathrm{MKW}}^{A}
\]
is the classical tree symmetrization operator of Munthe-Kaas and Wright \mcite{MKW2008}. Thus
three corners of the square are already available:
\[
\mathcal H_{\mathrm{LOT}}^{A},\quad
\mathcal H_{\mathrm{BCK}}^{A},\quad
\mathcal H_{\mathrm{MKW}}^{A}.
\]
What is missing is the planar noncommutative multi-index corner. In other words, the missing object
is precisely
\[
\mathcal H_{\mathrm{PLOT}}^{A}.
\]
The contribution of the present paper is therefore not merely to provide another realization of an
already known Hopf algebra. Its central purpose is to construct this missing object itself and to show
that, once it is inserted, the whole square becomes commutative. More precisely, we construct a Hopf algebra $\mathcal H_{\mathrm{PLOT}}^{A}$ together with a Hopf algebra isomorphism
\[
J:\mathcal H_{\mathrm{PLOT}}^{A}
\bij 
\mathcal H_{\mathrm{MKW}}^{A}
\]
and a word symmetrization operator
\[
\Omega_w:\mathcal H_{\mathrm{LOT}}^{A}\inj \mathcal H_{\mathrm{PLOT}}^{A},
\]
so that the diagram
\begin{equation}
\xymatrix{
 \mathcal{H}_{{\rm LOT}}^{A} \ar@{^{(}->}[d]_{\Omega_w}\ar@{^{(}->}[r]^{j}
 & \mathcal{H}_{{\rm BCK}}^{A}  \ar@{^{(}->}[d]^{\Omega_t}\\
 \mathcal{H}_{{\rm PLOT}}^{A} \ar[r]^{J}_{\sim}
 & \mathcal{H}_{{\rm MKW}}^{A}
}
\mlabel{eq:intro-square}
\end{equation}
commutes. Equivalently, the point is to prove the identity
\begin{equation}\label{eq:intro-commute}
J\circ \Omega_w=\Omega_t\circ j.
\end{equation}
This is the conceptual endpoint of the paper: the previously incomplete square is completed by the
construction of $\mathcal H_{\mathrm{PLOT}}^{A}$, and the four Hopf algebras are linked by the
commuting relation \eqref{eq:intro-commute}.\\

This motivation also explains our method. We do not start from a free
noncommutative Novikov algebra. Instead, we work directly in the free
associative algebra generated by a weighted alphabet $\mathcal Z_A:=\{z_i^a,\, i\in\mathbb Z_{\ge -1} \hbox{ and }a\in A\}$. On the non-positive weight part of the free associative algebra, we define an
insertion-type product $\blacktriangleright_w$.  We then prove that it restricts to the Lie
subalgebra $\mathfrak g(A)$ generated by a subspace $V(A)$ spanned by suitable monomials of weight $-1$, and yields a post-Lie
algebra structure on it \mcite{KMF2021,BK2023,FK2014, FARD2015,Foissy2018}.
Applying the Guin-Oudom construction then produces the Hopf algebra
$\mathcal H_{\mathrm{PLOT}}^{A}$.\\

Once this algebraic framework is in place, the next task is to connect it to planar rooted trees.
For that purpose we define a planar  tree  fertility map from decorated planar rooted trees to
monomials in $V(A)$. Here the planar situation exhibits a remarkable rigidity: the planar order
determines a canonical parent structure on the word side, and the planar tree fertility map turns out to be
bijective, with image $V(A)$. Thus, unlike the commutative case, each monomial in $V(A)$ corresponds to a unique
decorated planar rooted tree. This rigidity is exactly what makes the construction of the lower-left
corner in \eqref{eq:intro-square} both natural and workable.\\

The bijection provided by the planar tree fertility map allows us to transport the Munthe-Kaas--Wright
coproduct to the monomial side. A major point of the paper is that the transported coproduct is not
left in a purely abstract dual form. We reconstruct the parent map directly from a word, define
left-admissible cuts intrinsically on the monomial side, and obtain an explicit combinatorial formula
for the coproduct of $\mathcal H_{\mathrm{PLOT}}^{A}$ in terms of pruning and trunk factors written
entirely in the word model. We then establish the extraction-contraction coproduct and finally
construct the symmetrization operator $\Omega_w$, whose role is to make the square
\eqref{eq:intro-square} commute.\\

The preceding discussion identifies the missing lower-left corner in
\eqref{eq:intro-square}; we now summarize how the results of the paper construct
this corner and make the whole diagram commute.

\subsection{Main contributions and outline of the paper}
\label{subsec:main-contributions-outline}

The main contributions of the paper can be summarized as follows.

\begin{enumerate}
\item We construct the planar noncommutative multi-index object at the algebraic
level. More precisely, starting from the weighted noncommutative alphabet $\mathcal Z_A$, we
prove that the Lie algebra $\mathfrak g(A)$ carries a
natural post-Lie algebra structure (Theorem~\ref{postPN}). This is the basic
algebraic input which makes the planar LOT construction possible. Applying the
Guin--Oudom construction to this post-Lie algebra then produces the Hopf algebra
$\mathcal H_{\rm PLOT}^{A}$, which is the missing planar noncommutative
multi-index Hopf algebra in \eqref{eq:intro-square}.

\item We connect this newly constructed monomial model with the planar rooted-tree
model. The planar fertility map is shown to be compatible with the post-Lie
structures (Proposition~\ref{faen}), and the planar order gives a rigidity
phenomenon which does not occur in the commutative multi-index setting: every
weight $-1$ monomial in $V(A)$ determines a unique decorated planar rooted tree
(Proposition~\ref{pro:Phibi}). Consequently, the transpose map yields the Hopf
algebra  isomorphism
\[
J:\mathcal H_{\rm PLOT}^{A}\bij \mathcal H_{\rm MKW}^{A}
\]
(Equation~\ref{J:hopfm}), which identifies the newly constructed lower-left
corner with the planar rooted-forest side.

\item We give an intrinsic word-level formula for the coproduct of
$\mathcal H_{\rm PLOT}^{A}$. Instead of leaving the coproduct as a transported
structure from $\mathcal H_{\rm MKW}^{A}$, we reconstruct the parent map directly
from a word, define left-admissible cuts on noncommutative monomials, and obtain
a completely explicit pruning--trunk formula (Theorem~\ref{thm:plotcoprodladm}).
This shows that $\mathcal H_{\rm PLOT}^{A}$ is not merely an abstract copy of the
MKW Hopf algebra, but a genuine noncommutative multi-index Hopf algebra with its
own combinatorial cut calculus. We also establish the corresponding
extraction--contraction coproduct and prove its compatibility with the planar
rooted-forest extraction--contraction structure (Theorem~\ref{thm:ec-transportA}).

\item We complete the comparison with the commutative/nonplanar theory. On the
tree side, the classical symmetrization operator $\Omega_t$ embeds the BCK Hopf
algebra into the MKW Hopf algebra. We construct the corresponding word
symmetrization operator
\[
\Omega_w:\mathcal H_{\rm LOT}^{A}\longrightarrow \mathcal H_{\rm PLOT}^{A},
\]
and prove that it is compatible with $\Omega_t$, $j$, and $J$. Equivalently, the
square
\[
J\circ\Omega_w=\Omega_t\circ j
\]
commutes (Theorem~\ref{thm:symmetrization-square}). Thus the construction of
$\mathcal H_{\rm PLOT}^{A}$ does exactly what was missing from the existing
picture: it fills the planar noncommutative multi-index corner and links the four
Hopf algebras in one coherent commutative diagram.
\end{enumerate}

The organization of the paper reflects this progression.  Section~2 introduces the weighted noncommutative alphabet,
the insertion product $\btri$ on $\mathfrak g(A)$, and the corresponding post-Lie
algebra structure (Theorem~\ref{postPN}). It then recalls the
Munthe-Kaas--Wright Hopf algebra, constructs the planar tree fertility map, establishes the tree--monomial bijection (Proposition~\ref{pro:Phibi}), proves
its compatibility with the post-Lie structures (Proposition~\ref{faen}), and
obtains the Hopf algebra isomorphism
$J:\mathcal H_{\rm PLOT}^{A} \bij  \mathcal H_{\rm MKW}^{A}$. Section~3 reconstructs the parent map on the
monomial side and derives the explicit coproduct formula for
$\mathcal H_{\rm PLOT}^{A}$ in terms of left-admissible cuts
(Theorem~\ref{thm:plotcoprodladm}). Section~4 develops the
extraction--contraction coproduct and proves that it is transported from the
ordered planar rooted-forest side to the word side (Theorem~\ref{thm:ec-transportA}).
Finally, Section~5 constructs the word symmetrization operator $\Omega_w$ and
proves that it is compatible with the classical tree symmetrization operator
$\Omega_t$, thereby proving the commutative square
(Theorem~\ref{thm:symmetrization-square}).

\medskip

\noindent\textbf{Notations.}
Throughout the paper we work over a field $\bfk$ of characteristic zero.
All vector spaces, algebras, coalgebras, tensor products, and linear maps are
taken over $\bfk$.  Unless otherwise specified, an algebra means a unital
associative $\bfk$-algebra, and a coalgebra means a counital coassociative
$\bfk$-coalgebra.  We use Sweedler's notation for the coproduct:
\[
\Delta(X)=\sum_{(X)} X_{(1)}\otimes X_{(2)}.
\]

\section{Planar Linares-Otto-Tempelmayr Hopf algebra}\label{notation}
In this section, we first construct the planar Linares--Otto--Tempelmayr
Hopf algebra $\mathcal H_{\rm PLOT}^{A}$.  We then construct the planar
tree fertility map and identify decorated planar rooted trees with the monomials
in $V(A)$.  Finally, we derive the mock-cocycle condition of the
coproduct $\Delta_{\rm PLOT}$.

\subsection{Planar LOT Hopf algebra} \label{ssec: PLOT}
In this subsection, we introduce a weighted alphabet $\vz_A$ associated to a
set $A$ of decorations, and the free noncommutative associative unital
algebra $\bnca:=\bfk\langle \vz_A\rangle$ with its length and weight
gradings.  We define an insertion-type product $\blacktriangleright_w$ on the non-positive weight part $\bnca_{\le 0}$ of
$\bnca$, giving it a post-Lie algebra structure, and then restrict to the Lie subalgebra
$\mathfrak g(A)$ generated by a suitable subset of monomials of weight $-1$.  Applying the
Guin--Oudom construction and taking the graded linear dual gives the planar
LOT Hopf algebra.\\

\noindent  Let $A$ be a finite set of decorations, and consider the collection of variables
\[
\vz_A:=\{z_i^a\mid (a,i)\in A\times \mathbb Z_{\ge -1}\}.
\]
The integer $i\ge -1$ is called the  weight of the variable $z_i^a$.
Let $\bnca:=\bfk\langle \vz_A\rangle$ be the free unital associative algebra on $\vz_A$.
A monomial in $\bnca$ is a word of the form
\begin{equation}
\bzia= z_{i_1}^{a_1}\cdots z_{i_p}^{a_p} ,
\mlabel{eq:bzia}
\end{equation}
where each letter $x_j:=z_{i_j}^{a_j}\in \vz_A$  is allowed to occur multiple times. The integer $p$ in \eqref{eq:bzia} is the {\bf length} $\vert\bzia\vert$ of the word $\bzia$. This yields a grading
\[
\bnca=\bigoplus_{n\ge 0}\bnca^{(n)},
\]
where $$\bnca^{(n)}:= \bfk \{\bzia\in \bnca,\, \vert\bzia\vert= n\}.$$ 
We define the weight map and the decoration map:
\begin{alignat*}{2}
\wt:\vz_A &\to \mathbb Z_{\ge -1} \quad & \quad d:\vz_A &\to A\\
z_i^a &\mapsto i, & z_i^a &\mapsto a.
\end{alignat*}
The {\bf weight} of the monomial
$\bzia= z_{i_1}^{a_1}\cdots z_{i_p}^{a_p}$ is defined by
\[
\wt(\bzia):=\sum_{j=1}^p \wt(z_{i_j}^{a_j})=\sum_{j=1}^p i_j.
\]
We define
the non-positive weight part
\[
\bnca_{\le0}
:=
\bfk\{\bzia\in\bnca\mid \wt(\bzia)\le0\}
\]
and 
\[
\partial:\vz_A\to \vz_A,\quad z_i^a\mapsto z_{i+1}^a .
\]
By the Leibniz rule, it extends uniquely to a derivation
\[
\partial:\bnca \to \bnca,\quad
x_1\cdots x_p\mapsto
\sum_{j=1}^{p}
x_1\cdots x_{j-1}\,\partial(x_j)\,x_{j+1}\cdots x_p .
\]
For $m\ge0$, we write $\partial^m$ for the $m$-fold composition of
$\partial$, with the convention $\partial^0=\id$. The derivation $\partial$ is homogeneous of degree $+1$ with respect to the
weight, and homogeneous of degree zero with respect to the length.

Let us consider the Lie algebra $\left(\bnca_{\le0},\lpp{-}{-}\right)$ equipped with the commutator bracket $\lpp{-}{-}$. 
Set
\[
V(A):=
\bfk\Bigl\{
z^{a_1}_{i_1}\cdots z^{a_p}_{i_p}\in\bnca
\ \Big|\ 
\sum_{r=1}^{p}i_r=-1,\ 
\sum_{r=1}^{q}i_r\ge 0\ \text{for all }1\le q<p
\Bigr\}.
\]
Equivalently, $V(A)$ is the $\bfk$-vector subspace of $\bnca$ spanned by the
weight $-1$ monomials whose proper left prefixes all have nonnegative weight.
Here a proper left prefix of a word $x_1\cdots x_p$ means a word
$x_1\cdots x_q$ with $1\le q<p$.
Define
\[
\mathfrak g(A):=\mathrm{Lie}\big(V(A)\big)
\subseteq
\left(\bnca_{\le0},\lpp{-}{-}\right)
\]
to be the Lie subalgebra of $\left(\bnca_{\le0},\lpp{-}{-}\right)$ generated by $V(A)$. As a Lie subalgebra of the free Lie algebra $\left(\bnca,\lpp{-}{-}\right)$, the Lie algebra $\mathfrak g(A)$ is also a free thanks to the \v Sir\v sov-Witt theorem.\\

The commutator bracket $\lpp{-}{-}$ is homogeneous with respect to the length.
Hence $\mathfrak g(A)$ is also $\mathbb{Z}_{\geq 0}$-graded by the length:
\begin{equation}
\mathfrak g(A)=\bigoplus_{n\ge 0}\mathfrak g(A)^{(n)},\,\text{ where }\,
\mathfrak g(A)^{(n)}:=\mathfrak g(A)\cap \bnca^{(n)}.
\mlabel{eq:deggrad}
\end{equation}
Notice that each $\mathfrak g(A)^{(n)}$ is finite-dimensional.   Indeed, in a
length $n$ monomial of $V(A)$, every index belongs to the finite
interval $[-1,n-2]$; hence there are only finitely many such monomials, and
only finitely many iterated Lie words with total length $n$.\\

\begin{defn}~\cite{FK2014,MKL2013}
A {\bf post-Lie algebra} $(\bg, \blacktriangleright, [-,-])$  is a Lie algebra $(\bg, [-,-])$ together with a  product $\blacktriangleright : \bg \otimes \bg \to \bg$ such that, for any $x, y, z \in \bg$,
\begin{align*}
x \blacktriangleright [y, z] &= [x \blacktriangleright y, z] + [y, x \blacktriangleright z],\\ 
[x, y] \blacktriangleright z &= a_\blacktriangleright(x, y, z) - a_\blacktriangleright(y, x, z),\nonumber
\mlabel{eq:postlie1}
\end{align*}
where $a_\blacktriangleright(x, y, z) = x \blacktriangleright (y \blacktriangleright z) - (x \blacktriangleright y) \blacktriangleright z$ is the associator.
\end{defn}

\noindent Now we build a post-Lie product $\btri$ on $\bnca_{\leq 0}$.

\begin{definition}
Define
\begin{equation}
\begin{aligned}
\btri:\bnca_{\leq 0}\otimes \bnca_{\leq 0}
&\longrightarrow \bnca_{\leq 0},\\
\bx\otimes \by
&\longmapsto
\sum_{j=1}^q
y_1\cdots y_{j-1}
\partial^{-{\rm wt}(\bx)}(y_j)
\bx
y_{j+1}\cdots y_q .
\end{aligned}
\mlabel{eq:winser}
\end{equation}
where $\bx=x_1\cdots x_p$ and $\by=y_1\cdots y_q\in\bnca_{\leq 0}$. The formula is first given for monomials and then extended $\bfk$-bilinearly.
\end{definition}

The following example illustrates~\eqref{eq:winser}.

\begin{exam}
Let ${\bf x}=z_0z_{-1}$ and ${\bf y}=z_1z_0z_{-1}z_{-1}\in V(A)$, in the undecorated case $A=\{*\}$.
Then
\[
\begin{aligned}
{\bf x}\btri{\bf y}
={}&\ \partial(z_1)\,{\bf x}\,z_0z_{-1}z_{-1}
+z_1\,\partial(z_0)\,{\bf x}\,z_{-1}z_{-1}
+z_1z_0\,\partial(z_{-1})\,{\bf x}\,z_{-1}
+z_1z_0z_{-1}\,\partial(z_{-1})\,{\bf x}\\
={}&\ z_2z_0z_{-1}z_0z_{-1}z_{-1}
+z_1z_1z_0z_{-1}z_{-1}z_{-1}
+z_1z_0z_0z_0z_{-1}z_{-1}
+z_1z_0z_{-1}z_0z_0z_{-1}.
\end{aligned}
\]
\end{exam}

\begin{remark}
\begin{enumerate}
\item The product $\btri$ is defined by  \meqref{eq:winser},
where $\partial^{-\mathrm{wt}(\bx)}$ shifts the index of the $j$-th letter of $\by$ by $-\mathrm{wt}(\bx)$.
So the product $\bx \btri \by$ encodes inserting $\bx$ after every letter of $\by$ via the operator $\partial^{-\mathrm{wt}(\bx)}$ applied to the letter.
The case $\mathrm{wt}(\bx)=-1$ is of particular interest:
$$\partial^{-\rm wt(\bx)}(y_{j}) =: \partial(z_{i_{j}}^b) = z_{i_{j}+1}^b,$$
meaning that the letter $z_{i_{j}}^b$ of $\by$ is replaced by the letter $z_{i_{j}+1}^b$, and then $\bx$ is inserted immediately after it.
\mlabel{it:PNsta}

\item
The index shift via $\partial^{-{\rm wt}(\bx)}$ ensures that the linear endomorphisms ${\bf x}\btri -$ of $\bnca_{\le 0}$ respect the weight. Indeed, the weight of any summand in the right hand side of~\meqref{eq:winser} is
\begin{align}
\Bigl(\sum_{\ell=1}^{j-1} {\rm wt(y_{\ell})}\Bigr) + \big({\wt(y_{j})}-{\rm wt}(\bx)\big) + {\rm wt}(\bx) + \Bigl(\sum_{\ell=j+1}^{q} {\rm wt(y_{\ell})}\Bigr)
&=\sum_{\ell=1}^q {\rm wt(y_{\ell})}\nonumber\\
&={\rm wt}(\by).\nonumber
\end{align}

\end{enumerate}
\mlabel{rem:PNstabletri}
\end{remark}

\begin{lemma}
The triple $(\bnca_{\leq 0},\btri,\lpp{-}{-})$  is a post-Lie algebra.
\mlabel{lem:ncpost}
\end{lemma}

\begin{proof}
Let
\[
\bx:=x_1\cdots x_p,
\quad
\by=y_1\cdots y_q,
\quad
\bz:=z_1\cdots z_r,
\]
in $\bnca_{\leq 0}$, where $x_i,y_j,z_k\in \vz_A$.
For $1\le k\le q$, set
\[
\by_{<k}:=y_1\cdots y_{k-1},
\quad
\by_k:=y_k,
\quad
\by_{>k}:=y_{k+1}\cdots y_q,
\]
and similarly define $\bz_{<l}$, $\bz_l$ and $\bz_{>l}$ for $\bz$.

\noindent\noindent{\bf Step 1.} We prove that
\begin{equation}
\bx \btri \lpp{\by}{\bz}
=
\lpp{\bx\btri \by}{\bz}
+
\lpp{\by}{\bx\btri \bz}.
\mlabel{eq:aim1}
\end{equation}
By~\eqref{eq:winser},
\[
\bx\btri(\by\bz)
=
\sum_{k=1}^{q+r}(\by\bz)_{<k}\,\partial^{-\rm wt(\bx)}\bigl((\by\bz)_k\bigr)\,\bx\,(\by\bz)_{>k}.
\]
For $1\le k\le q$, one has
\[
(\by\bz)_{<k}=\by_{<k},
\quad
(\by\bz)_k=y_k,
\quad
(\by\bz)_{>k}=\by_{>k}z_1\cdots z_r.
\]
For $q<k\le q+r$, writing $l:=k-q$, one has
\[
(\by\bz)_{<k}=y_1\cdots y_q\bz_{<l},
\quad
(\by\bz)_k=z_l,
\quad
(\by\bz)_{>k}=\bz_{>l}.
\]
Hence
\begin{align}
\bx\btri(\by\bz)
&=
\sum_{k=1}^q \by_{<k}\partial^{-\rm wt(\bx)}(y_k)\,\bx\,\by_{>k}z_1\cdots z_r
+
\sum_{l=1}^r y_1\cdots y_q\bz_{<l}\partial^{-\rm wt(\bx)}(z_l)\,\bx\,\bz_{>l}\nonumber\\
&=(\bx\btri\by)\bz+\by(\bx\btri\bz).
\mlabel{eq:xyz}
\end{align}
Similarly,
\begin{equation}
\bx\btri(\bz\by)=(\bx\btri\bz)\by+\bz(\bx\btri\by).\mlabel{eq:xzy}
\end{equation}
Subtracting~\eqref{eq:xzy} from~\eqref{eq:xyz}, we obtain
\begin{equation*}
\bx\btri \lpp{\by}{\bz}
=\lpp{\bx\btri \by}{\bz}+\lpp{\by}{\bx\btri \bz},
\mlabel{eq:laim1}
\end{equation*}
which proves~\eqref{eq:aim1}.

\noindent\noindent{\bf Step 2.} We show that
\begin{equation}
\lpp{\bx}{\by}\btri \bz
=
a_{\btri}(\bx,\by,\bz)-a_{\btri}(\by,\bx,\bz).
\mlabel{eq:aim2}
\end{equation}
For the left-hand side, it follows from~\eqref{eq:winser} that
\begin{align}
\lpp{\bx}{\by}\btri \bz
=
(\bx\by-\by\bx)\btri \bz
=
\sum_{l=1}^{r}\bz_{<l}\partial^{-\rm wt(\bx\by)}(z_l)(\bx\by-\by\bx)\bz_{>l}.
\mlabel{eq:lhsfinal}
\end{align}
For the right-hand side, we first have
\begin{align}
\bx \btri (\by \btri \bz)
=&\ \bx \btri \biggl(\sum_{l=1}^{r}\bz_{<l}\partial^{-\rm wt(\by)}(z_l)\by\bz_{>l}\biggr)
\overset{\eqref{eq:winser}}{=}
\sum_{l=1}^{r}\bx \btri \bigl(\bz_{<l}\partial^{-\rm wt(\by)}(z_l)\by\bz_{>l}\bigr)
\nonumber\\
=&\ \sum_{l=1}^{r}\sum_{j=1}^{l-1}
\bz_{<j}\partial^{-\rm wt(\bx)}(z_j)\bx z_{j+1}\cdots z_{l-1}\partial^{-\rm wt(\by)}(z_l)\by\bz_{>l}
\label{eq:x-tri-yz-1}
\\
&\ +\sum_{l=1}^{r}
\bz_{<l}\partial^{-\rm wt(\bx\by)}(z_l)\bx\by\bz_{>l}\nonumber
\\
&\ +\sum_{l=1}^{r}\sum_{m=1}^{q}
\bz_{<l}\partial^{-\rm wt(\by)}(z_l)\by_{<m}\partial^{-\rm wt(\bx)}(y_m)\bx\by_{>m}\bz_{>l}
\label{eq:x-tri-yz-3}
\\
&\ +\sum_{l=1}^{r}\sum_{t=1}^{r-l}
\bz_{<l}\partial^{-\rm wt(\by)}(z_l)\by z_{l+1}\cdots z_{l+t-1}\partial^{-\rm wt(\bx)}(z_{l+t})\bx\bz_{>l+t}.\nonumber
\end{align}
Similarly, we obtain
\begin{align}
(\bx\btri \by)\btri \bz
\overset{\eqref{eq:winser}}{=}&\ \sum_{l=1}^{r}\sum_{k=1}^{q}
\bz_{<l}\partial^{-\rm wt(\bx \btri \by)}(z_l)\by_{<k}\partial^{-\rm wt(\bx)}(y_k)\bx\by_{>k}\bz_{>l}.
\mlabel{eq:xytriz}
\end{align}
Comparing~\eqref{eq:x-tri-yz-3} and~\eqref{eq:xytriz},
\begin{align}
a_{\btri}(\bx,\by,\bz)
=&\ \sum_{l=1}^{r}\sum_{j=1}^{l-1}
\bz_{<j}\partial^{-\rm wt(\bx)}(z_j)\bx z_{j+1}\cdots z_{l-1}\partial^{-\rm wt(\by)}(z_l)\by\bz_{>l}
\label{eq:a-xyz-1}
\\
&\ +\sum_{l=1}^{r}
\bz_{<l}\partial^{-\rm wt(\bx\by)}(z_l)\bx\by\bz_{>l}
\label{eq:a-xyz-2}
\\
&\ +\sum_{l=1}^{r}\sum_{t=1}^{r-l}
\bz_{<l}\partial^{-\rm wt(\by)}(z_l)\by z_{l+1}\cdots z_{l+t-1}\partial^{-\rm wt(\bx)}(z_{l+t})\bx\bz_{>l+t}.
\label{eq:a-xyz-3}
\end{align}
By symmetry,
\begin{align}
a_{\btri}(\by,\bx,\bz)
=&\ \sum_{l=1}^{r}\sum_{j=1}^{l-1}
\bz_{<j}\partial^{-\rm wt(\by)}(z_j)\by z_{j+1}\cdots z_{l-1}\partial^{-\rm wt(\bx)}(z_l)\bx\bz_{>l}
\label{eq:a-yxz-1}
\\
&\ +\sum_{l=1}^{r}
\bz_{<l}\partial^{-\rm wt(\bx\by)}(z_l)\by\bx\bz_{>l}
\label{eq:a-yxz-2}
\\
&\ +\sum_{l=1}^{r}\sum_{t=1}^{r-l}
\bz_{<l}\partial^{-\rm wt(\bx)}(z_l)\bx z_{l+1}\cdots z_{l+t-1}\partial^{-\rm wt(\by)}(z_{l+t})\by\bz_{>l+t}.
\label{eq:a-yxz-3}
\end{align}

Now we compare the two  double-sum terms in \eqref{eq:a-xyz-1} and \eqref{eq:a-xyz-3}.
In~\eqref{eq:a-xyz-1}, setting $t=l-j$, $\eqref{eq:a-xyz-1}$ be written as
\[
\sum_{j=1}^{r}\sum_{t=1}^{r-j}
\bz_{<j}\partial^{-\rm wt(\bx)}(z_j)\bx z_{j+1}\cdots z_{j+t-1}\partial^{-\rm wt(\by)}(z_{j+t})\by\bz_{>j+t},
\]
which is exactly~\eqref{eq:a-yxz-3}.
Likewise, reindexing~\eqref{eq:a-xyz-3} shows that the difference between the two terms is 0 in $\eqref{eq:a-xyz-3}$ and $\eqref{eq:a-yxz-1}$.

Therefore we conclude that
\begin{align*}
a_{\btri}(\bx,\by,\bz)-a_{\btri}(\by,\bx,\bz)
\overset{\eqref{eq:a-xyz-2}-\eqref{eq:a-yxz-2}}{=}\sum_{l=1}^{r}\bz_{<l}\partial^{-\rm wt(\bx\by)}(z_l)(\bx\by-\by\bx)\bz_{>l},
\end{align*}
which is exactly the left-hand side of~\meqref{eq:aim2} by~\eqref{eq:lhsfinal}.
This completes the proof of~\meqref{eq:aim2} and the whole proof.
\end{proof}

\begin{lemma}\mlabel{lem:PN-right-generator}
 Restricting~\meqref{eq:winser}  yields a map
\begin{equation}
\btri:\mathfrak g(A)\otimes V(A)\longrightarrow V(A),
\quad
u\otimes x\longmapsto u\btri x.
\end{equation}
In particular, we have
\begin{equation}
\btri:V(A)\otimes V(A)\longrightarrow V(A).
\mlabel{eq:winser1}
\end{equation}
\end{lemma}

\begin{proof}
It suffices to prove the first statement.  By bilinearity, we may assume
that $x$ is a monomial in $V(A)$ and that $u$ is an iterated Lie
monomial in monomial elements
\[
u_1,\ldots,u_\ell\in V(A).
\]
Since $\mathfrak g(A)$ is realized as a Lie subalgebra of the associative
algebra $\bnca$ equipped with the commutator bracket, this iterated Lie
monomial expands in $\bnca$ as a finite $\bfk$-linear combination of
associative words of the form
\[
\overline{w}
=
u_{\sigma(1)}\cdots u_{\sigma(\ell)}
\]
for suitable permutations $\sigma$.  Thus it suffices to prove
\[
\overline{w}\btri x\in V(A)
\]
for every such word $\overline{w}$ and every monomial $x \in V(A)$. Fix such a word and write $\overline{w}=u_1\cdots u_\ell$.  Since each
$u_i$ has weight $-1$, we have
\begin{equation}
\wt(\overline{w})=-\ell.
\mlabel{eq:right-generator-total-weight-wbar}
\end{equation}
Moreover, every nonempty proper left prefix of $\overline{w}$ has weight at
least $1-\ell$.  Indeed, such a prefix has the form
\[
\overline{w}'=u_1\cdots u_{m-1}u'_m,
\]
where either $u'_m$ is a nonempty proper left prefix of $u_m$, or
$u'_m=u_m$ and $m<\ell$.  In the first case, $\wt(u'_m)\ge0$, so
\[
\wt(\overline{w}')=-(m-1)+\wt(u'_m)\ge1-\ell.
\]
In the second case,
\[
\wt(\overline{w}')=-m\ge1-\ell.
\]
Now write
\[
x=y_1\cdots y_q\in V(A).
\]
For $1\le s\le q$, set
\[
P_x(s):=\sum_{r=1}^{s}\wt(y_r).
\]
Since $x\in V(A)$, the proper-prefix and total-weight conditions give
\begin{equation}
P_x(s)\ge0 \quad \text{for } 1\le s<q,
\quad
P_x(q)=-1.
\mlabel{eq:right-generator-prefix-x}
\end{equation}

\noindent By \eqref{eq:winser} and \eqref{eq:right-generator-total-weight-wbar},
$\overline{w}\btri x$ is a sum of words
\[
V_j
:=
y_1\cdots y_{j-1}\,
\partial^{\ell}(y_j)\,
\overline{w}\,
y_{j+1}\cdots y_q,
\quad
1\le j\le q.
\]
We prove that every $V_j$ belongs to $V(A)$.  First,
\[
\wt(V_j)=\wt(x)+\ell+\wt(\overline{w})=-1+\ell-\ell=-1.
\]
It remains to check the proper-prefix condition.  Let $V'$ be a nonempty
proper left prefix of $V_j$.  According to the position where $V'$ ends,
there are four cases.

\smallskip
\noindent{\bf Case 1.} The prefix $V'$ ends before the shifted letter
$\partial^\ell(y_j)$.  Then
\[
V'=y_1\cdots y_s
\]
for some $1\le s<j$.  Since $s<q$, \eqref{eq:right-generator-prefix-x}
gives
\[
\wt(V')=P_x(s)\ge0.
\]

\smallskip
\noindent{\bf Case 2.} The prefix $V'$ ends exactly at the shifted letter
$\partial^\ell(y_j)$.  Then
\[
V'=y_1\cdots y_{j-1}\partial^\ell(y_j),
\]
and hence
\[
\wt(V')=P_x(j)+\ell.
\]
If $j<q$, then $P_x(j)\ge0$ by \eqref{eq:right-generator-prefix-x}.  If
$j=q$, then
\[
\wt(V')=P_x(q)+\ell=-1+\ell\ge0.
\]
Thus $\wt(V')\ge0$.

\smallskip
\noindent{\bf Case 3.} The prefix $V'$ ends inside the inserted word
$\overline{w}$.  Then
\[
V'=y_1\cdots y_{j-1}\partial^\ell(y_j)\overline{w}',
\]
where $\overline{w}'$ is a nonempty proper left prefix of $\overline{w}$.
By the prefix bound for $\overline{w}$,
\[
\wt(V')
=
P_x(j)+\ell+\wt(\overline{w}')
\ge
P_x(j)+1.
\]
If $j<q$, then $P_x(j)\ge0$.  If $j=q$, then $P_x(q)+1=0$.  Hence
$\wt(V')\ge0$.

\smallskip
\noindent{\bf Case 4.} The prefix $V'$ ends at or after the whole inserted
word $\overline{w}$ but before the end of $V_j$.  Then
\[
V'=y_1\cdots y_{j-1}\partial^\ell(y_j)
\overline{w}
y_{j+1}\cdots y_s
\]
for some $j\le s<q$.  Using $\wt(\overline{w})=-\ell$, we obtain
\[
\wt(V')=P_x(s)\ge0
\]
by \eqref{eq:right-generator-prefix-x}.

\smallskip
In all cases, every nonempty proper left prefix of $V_j$ has nonnegative
weight.  Since also $\wt(V_j)=-1$, we have $V_j\in V(A)$ for every
$1\le j\le q$.  Hence
\[
\overline{w}\btri x\in V(A).
\]
By the initial reduction and the bilinearity of $\btri$, the result follows
for all $u\in\mathfrak g(A)$ and all $x\in V(A)$.
\end{proof}

\begin{remark}
The subspace $V(A)$ is stable under $\btri$ by~\eqref{eq:winser1}.
However, it is not stable under the commutator bracket $\lpp{-}{-}$.
\mlabel{bncalie}
\end{remark}

\begin{lemma}
Restricting~\meqref{eq:winser} to $\mathfrak g(A)$ yields a product
\begin{equation}
\btri:\mathfrak g(A)\otimes\mathfrak g(A)\longrightarrow\mathfrak g(A).
\mlabel{eq:btri}
\end{equation}
\end{lemma}
\begin{proof}
Recall that $\mathfrak g(A)$ is the Lie subalgebra of $\left(\bnca_{\leq 0},\lpp{-}{-}\right)$ generated by $V(A)$.
By Lemma~\mref{lem:ncpost} together with~\eqref{eq:winser1}, we have
\[
x \btri \lpp{y}{z}
=
\lpp{x\btri y}{z}+\lpp{y}{x\btri z}\in \mathfrak g(A),
\quad
\forall\, x,y,z\in V(A).
\]
Therefore, the product $\btri$ in \eqref{eq:winser1} can be extended to
\begin{equation}
\btri:V(A)\otimes \mathfrak g(A)\to \mathfrak g(A).
\mlabel{eq:1step}
\end{equation}

We next extend~\eqref{eq:1step} to~\eqref{eq:btri} by induction on the number of iterated Lie brackets appearing in the first multi-factor of $\blacktriangleright_{w}$. More precisely, the induction is on the bracket depth of the first argument.
In the inductive step, we prove the assertion for a bracket
$\lpp{x}{\by}$ with $x\in V(A)$ and with $\by$ of
strictly smaller bracket depth.
The initial step is precisely~\eqref{eq:1step}.
For the inductive step, let $x\in V(A)$, let $\by\in\mathfrak g(A)$ be of smaller bracket depth, and let $\bz\in\mathfrak g(A)$.
Applying Lemma~\mref{lem:ncpost} once again, we obtain
\begin{align}
\lpp{x}{\by}\btri \bz
={}&
x\btri (\by\btri\bz)
-(x\btri \by)\btri \bz
-\by\btri (x\btri \bz)
+(\by\btri x)\btri \bz.
\end{align}
Notice that
\[
\by\btri \bz,\,
x\btri \by,\,
x\btri \bz
\in \mathfrak g(A).
\]
The first statement comes from  the  induction hypothesis, the second and the third one come  from ~\eqref{eq:1step}.
Moreover, by Lemma~\mref{lem:PN-right-generator}, we have
\begin{equation}
\by\btri x\in V(A).
\mlabel{eq:fourth}
\end{equation}
It follows that the four terms $x\btri (\by\btri \bz)$, $(x\btri \by)\btri \bz$, $\by\btri (x\btri \bz)$, and $(\by\btri x)\btri \bz$ belong to $\mathfrak g(A)$. Here, the first  statement comes from~\eqref{eq:1step}, the second and the third come from the induction hypothesis, while the fourth  comes from~\eqref{eq:fourth} and~\eqref{eq:1step}.
Therefore,
\[
\lpp{x}{\by}\btri \bz\in \mathfrak g(A),
\]
which completes the inductive step.
\end{proof}

We can now combine the above  post-Lie identities with the closure result
just proved. This gives the desired post-Lie algebra on the planar multi-index
side.

\begin{theorem}
The triple $\left(\mathfrak g(A),\btri,\lpp{-}{-}\right)$ is a post-Lie algebra.
\mlabel{postPN}
\end{theorem}

\begin{proof}
Since $\mathfrak g(A)\subseteq \bnca_{\le0}$ is a Lie subalgebra, the two
post-Lie identities in Lemma~\mref{lem:ncpost} hold for elements of
$\mathfrak g(A)$.  \eqref{eq:btri} shows that $\btri$ is closed on
$\mathfrak g(A)$.  Hence $\left(\mathfrak g(A),\btri,\lpp{-}{-}\right)$ is a post-Lie algebra.
\end{proof}

Now let us recall the Guin-Oudom construction~\cite{FK2014} for post-Lie algebras.
Let $(\bg,\blacktriangleright,\lpp{-}{-})$ be a post-Lie algebra, and let $U(\bg)$ denote its universal enveloping algebra.
The post-Lie product $\blacktriangleright$ extends uniquely to a bilinear map
\[
\blacktriangleright: U(\bg)\otimes U(\bg)\to U(\bg),
\]
characterized by
\begin{align*}
\etree \blacktriangleright X =&\ X,\quad
(xX)\blacktriangleright y
= x\blacktriangleright (X\blacktriangleright y)
   - (x\blacktriangleright X)\blacktriangleright y,\\
X\blacktriangleright (YZ)
=&\ \sum_{(X)}
   \bigl(X_{(1)}\blacktriangleright Y\bigr)
   \bigl(X_{(2)}\blacktriangleright Z\bigr),
\quad
\forall\, x,y\in\bg,\; X,Y,Z\in U(\bg).
\end{align*}
The associated Grossman--Larson product on $U(\bg)$ is then defined by
\[
X\star Y:=X_{(1)}\bigl(X_{(2)}\blacktriangleright Y\bigr).
\]
With this product, $(U(\bg),\star,\Delta_{\shuffle},\etree,\varepsilon)$ becomes a Hopf algebra. Applying the above Guin--Oudom construction to the post-Lie algebra
$(\mathfrak g(A),\btri,\lpp{-}{-})$
 yields a Hopf algebra
\[
\Big(U\big(\mathfrak g(A)\big),\star,\Delta_{\shuffle},\etree,\varepsilon\Big).
\]

The grading in \meqref{eq:deggrad} extends naturally to $U\big(\mathfrak g(A)\big)$, with each graded component being finite-dimensional.
Hence, taking the graded linear dual gives rise to a Hopf algebra
\begin{equation*}
\mathcal H_{\rm PLOT}^{A}
:=
\Big(U\big(\mathfrak g(A)\big),\star,\Delta_{\shuffle},\etree,\varepsilon\Big)^{g}.
\end{equation*}

\subsection{Planar fertility map}\label{sec:order}
This subsection is devoted to the planar fertility map from the MKW Hopf algebra to the planar LOT Hopf algebra.
Let us first establish a linear order on the vertex set of a planar rooted tree.

Denote by $\OT^A$ (resp. $\OF^A$) the set of $A$-decorated planar rooted trees (resp. forests).
For an $A$-decorated planar rooted tree $t\in \OT^A$, we define an order on the vertex set $V(t)$ recursively on ${\rm dep}(t)\geq 0$.
For the initial step of ${\rm dep}(t) = 0$, we have $t=\bullet_a$ for some $a\in A$ and the ordering is immediate.
For the inductive step of ${\rm dep}(t) \geq 1$, suppose $t=B_a^+(t_1,\ldots, t_m)$, where $B_a^+$ is the grafting operator given in~\meqref{eq:bplus} below.
We order the vertices of $t$ by listing the root of $t$ first, followed by the vertices of the subtrees $t_1, \ldots, t_m$ in sequence, where each subtree is ordered recursively according to this same rule. We call this order the \textbf{root-subtree order} on $V(t)$. This total order on the vertices is obtained by visiting them from left to right starting from the root, as shown below on an example.
\vskip-2mm
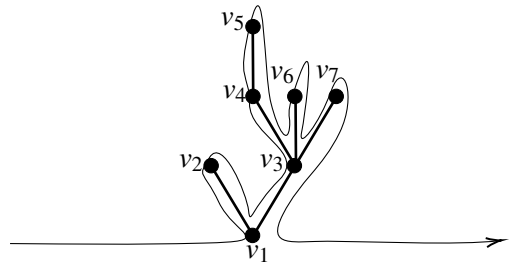
\begin{equation*}\label{total-order}
\begin{tikzpicture}[x=0.75pt,y=0.75pt,yscale=-0.7,xscale=0.7]
\draw  [fill={rgb, 255:red, 0; green, 0; blue, 0 }  ,fill opacity=1 ] (329.67,334.83) .. controls (329.67,331.98) and (331.98,329.67) .. (334.83,329.67) .. controls (337.69,329.67) and (340,331.98) .. (340,334.83) .. controls (340,337.69) and (337.69,340) .. (334.83,340) .. controls (331.98,340) and (329.67,337.69) .. (329.67,334.83) -- cycle ;
\draw  [fill={rgb, 255:red, 0; green, 0; blue, 0 }  ,fill opacity=1 ] (359.67,284.83) .. controls (359.67,281.98) and (361.98,279.67) .. (364.83,279.67) .. controls (367.69,279.67) and (370,281.98) .. (370,284.83) .. controls (370,287.69) and (367.69,290) .. (364.83,290) .. controls (361.98,290) and (359.67,287.69) .. (359.67,284.83) -- cycle ;
\draw  [fill={rgb, 255:red, 0; green, 0; blue, 0 }  ,fill opacity=1 ] (300,284.83) .. controls (300,281.98) and (302.31,279.67) .. (305.17,279.67) .. controls (308.02,279.67) and (310.33,281.98) .. (310.33,284.83) .. controls (310.33,287.69) and (308.02,290) .. (305.17,290) .. controls (302.31,290) and (300,287.69) .. (300,284.83) -- cycle ;
\draw  [fill={rgb, 255:red, 0; green, 0; blue, 0 }  ,fill opacity=1 ] (330,235.17) .. controls (330,232.31) and (332.31,230) .. (335.17,230) .. controls (338.02,230) and (340.33,232.31) .. (340.33,235.17) .. controls (340.33,238.02) and (338.02,240.33) .. (335.17,240.33) .. controls (332.31,240.33) and (330,238.02) .. (330,235.17) -- cycle ;
\draw  [fill={rgb, 255:red, 0; green, 0; blue, 0 }  ,fill opacity=1 ] (360,235.17) .. controls (360,232.31) and (362.31,230) .. (365.17,230) .. controls (368.02,230) and (370.33,232.31) .. (370.33,235.17) .. controls (370.33,238.02) and (368.02,240.33) .. (365.17,240.33) .. controls (362.31,240.33) and (360,238.02) .. (360,235.17) -- cycle ;
\draw  [fill={rgb, 255:red, 0; green, 0; blue, 0 }  ,fill opacity=1 ] (389.67,235.17) .. controls (389.67,232.31) and (391.98,230) .. (394.83,230) .. controls (397.69,230) and (400,232.31) .. (400,235.17) .. controls (400,238.02) and (397.69,240.33) .. (394.83,240.33) .. controls (391.98,240.33) and (389.67,238.02) .. (389.67,235.17) -- cycle ;
\draw  [fill={rgb, 255:red, 0; green, 0; blue, 0 }  ,fill opacity=1 ] (330,185.17) .. controls (330,182.31) and (332.31,180) .. (335.17,180) .. controls (338.02,180) and (340.33,182.31) .. (340.33,185.17) .. controls (340.33,188.02) and (338.02,190.33) .. (335.17,190.33) .. controls (332.31,190.33) and (330,188.02) .. (330,185.17) -- cycle ;
\draw    (305,284.67) -- (335.33,335) [line width=1] ;
\draw    (335.67,235.33) -- (366,285.67) [line width=1] ;
\draw    (365.67,235.33) -- (366,285) [line width=1] ;
\draw    (335,186) -- (335.33,235.67) [line width=1] ;
\draw    (395.33,235) -- (365.33,285) [line width=1] ;
\draw    (365.33,285) -- (335.33,335) [line width=1] ;
\draw    (161,340.67) .. controls (209.33,341) and (246.67,341.67) .. (300,340.33) .. controls (353.33,339) and (322,335) .. (310.67,311) .. controls (299.33,287) and (294.67,294.33) .. (304,279) .. controls (313.33,263.67) and (330.67,305) .. (332,317) .. controls (333.33,329) and (334,315) .. (354,295) .. controls (374,275) and (332.67,263.67) .. (332.67,240.33) .. controls (332.67,217) and (322.67,226.33) .. (333.33,183) .. controls (344,139.67) and (342.67,218.33) .. (348,239.67) .. controls (353.33,261) and (360,269.67) .. (361.33,257.67) .. controls (362.67,245.67) and (356.67,252.33) .. (368.67,219.67) .. controls (380.67,187) and (372.67,247.67) .. (370,261) .. controls (367.33,274.33) and (382.67,255) .. (391.33,231) .. controls (400,207) and (410,235) .. (397.33,261) .. controls (384.67,287) and (369.33,295) .. (356.67,319.67) .. controls (344,344.33) and (364.67,339) .. (388.67,339) .. controls (412.19,339) and (484.37,342.2) .. (511.72,338.56) ;
\draw [shift={(513.33,338.33)}, rotate = 171.25] [color={rgb, 255:red, 0; green, 0; blue, 0 }  ][line width=0.75]    (10.93,-3.29) .. controls (6.95,-1.4) and (3.31,-0.3) .. (0,0) .. controls (3.31,0.3) and (6.95,1.4) .. (10.93,3.29)   ;
\draw (311.67,225.07) node [anchor=north west][inner sep=0.75pt]    {$v_{4}$};
\draw (328.33,340.07) node [anchor=north west][inner sep=0.75pt]    {$v_{1}$};
\draw (280.33,274.4) node [anchor=north west][inner sep=0.75pt]    {$v_{2}$};
\draw (338.33,274.4) node [anchor=north west][inner sep=0.75pt]    {$v_{3}$};
\draw (310.33,173.07) node [anchor=north west][inner sep=0.75pt]    {$v_{5}$};
\draw (345.67,211.4) node [anchor=north west][inner sep=0.75pt]    {$v_{6}$};
\draw (378.33,210.73) node [anchor=north west][inner sep=0.75pt]    {$v_{7}$};
\end{tikzpicture}
\end{equation*}

\noindent The Munthe-Kaas--Wright (MKW) Hopf algebra
\[
(\mathcal{H}_{\mathrm{MKW}}^{A},\shuffle,\Delta_{\mathrm{MKW}},\etree,\varepsilon)
\]
is defined on $\bfk\OF^A$, see~\mcite{LMK2013,MKW2008}.
Its product is given by the shuffle product $\shuffle$ of planar rooted trees.
The coproduct $\Delta_{\mathrm{MKW}}$ is first defined on a planar rooted tree $t\in\OT^A$ in terms of the combinatorial notion of left-admissible cuts:
\begin{equation}
  \Delta_{\mathrm{MKW}}(t)=\sum_{c\in \mathrm{LAdm}(t)} P^{c}(t)\otimes R^{c}(t).
\mlabel{eq:mkwe}
\end{equation}
Here, for a left-admissible cut $c$, the term $P^{c}(t)$ denotes the forest consisting of the branches removed by the cut, while $R^{c}(t)$ denotes the remaining trunk. Among the subtrees cut off from the same vertex, the planar order is preserved; by contrast, the forests cut off from different vertices are shuffled together to form $P^{c}(t)$. The coproduct in~\meqref{eq:mkwe} also admits a recursive formula with respect to the number of vertices \mcite{Rahm2023,Rahm2021}
\begin{align}
\Delta_{\mathrm{MKW}}(t)=t\otimes \etree+ (\id\otimes B_{a}^{+})\Delta_{\mathrm{MKW}}\big(B_{a}^{-}(t)\big),\quad \forall a\in A.
\mlabel{eq:cocyle}
\end{align}
The coproduct in~\meqref{eq:mkwe} is extended to a forest $f\in\bf\OF^A$ by setting
\[
\Delta_{\mathrm{MKW}}(f)
=
(\mathrm{id}\otimes B^{-})
\Big(
\Delta_{\mathrm{MKW}}\big(B^{+}_{a}(f)\big)
-
B^{+}_{a}(f)\otimes \etree
\Big),
\quad \forall a\in A,
\]
where
\begin{equation}
B^{+}_{a}:\bfk\OF^A\to\bfk\OT^A
\mlabel{eq:bplus}
\end{equation}
is the grafting operator that attaches all trees in the input forest to a new root decorated by $a\in A$, and
 $B_a^{-}:\bfk\OT^A\to \bfk\OF^A$ is the transpose of $B_a^{+}$
\begin{equation*}
\langle B_a^{+}(f),\,g\rangle_{\mathcal F}
=
\langle f,\,B_a^{-}(g)\rangle_{\mathcal F},
\quad
\forall\,f,g\in\OF^A .
\end{equation*}
Here we employ the natural pairing
\begin{equation*}
\langle f,g\rangle_{\mathcal F}:=\delta_{f,g},
\qquad \forall\, f,g\in \OF^A.
\end{equation*}

\noindent The following example illustrates the coproduct.

\begin{exam}
{%
\thinmuskip=2mu\relax
\medmuskip=3mu\relax
\thickmuskip=4mu\relax
\[
\begin{aligned}
\Delta_{\mathrm{MKW}}\!\left(\Tbcad\right)
={}&
\etree \otimes \Tbcad
+\Tsingle{b} \otimes \Tacd
+\Tsingle{d} \otimes \Tabc
+\Tsingle{b} \shuffle \Tsingle{d} \otimes \Tac \\
&\quad
+\Tsingle{b}\,\Tcd \otimes \Tsingle{a}
+\Tbcad \otimes \etree.
\end{aligned}
\]
}
\end{exam}

\noindent The MKW Hopf algebra
\[(\mathcal{H}_{\mathrm{MKW}}^A, \shuffle, \Delta_{\mathrm{MKW}},\etree,\varepsilon)\]
is connected and graded by the number of vertices, which  admits a connected and graded dual Hopf algebra
\[\big((\mathcal{H}_{\mathrm{MKW}}^A)^{g}, \star, \Delta_\shuffle,1^{\ast},\varepsilon\big).\]
Elements in $(\mathcal{H}_{\mathrm{MKW}}^A)^{g}$ can simply be identified with the linear combinations of $\OF^A$, and the grading is again by the number of vertices.
 Here the coproduct $\Delta_\shuffle$ is the deshuffle coproduct, and the product $\star$ is the planar Grossman-Larson product given by
\[f\star g:= B^{-} \big(f\curvearrowright_{l} B^{+}_{a}(g)\big), \quad \forall f,g\in \OF^A, a\in A,\]
where $\curvearrowright_{l}$ is the left grafting operation.
In more detail, the left grafting operation on planar rooted trees is
\begin{equation}
\curvearrowright_{l}:\bfk \OT^A \otimes \bfk \OT^A  \to \bfk \OT^A, \quad t_1\otimes t_2 \mapsto t_{1}\curvearrowright_{l} t_{2} :=
\sum_{v_k\in V(t_{2})} t_{1}\curvearrowright_{v_k} t_{2}.
\mlabel{eq:lgraft}
\end{equation}
The above $\curvearrowright_{l}$ can be extended to
$$\curvearrowright_{l}:\bfk \OF^A \otimes \bfk \OT^A  \to \bfk \OT^A, \quad t_1\cdots t_n\otimes t \mapsto  t_1\cdots t_n\curvearrowright_{l} t,$$
where $t_1\cdots t_n\curvearrowright_{l} t$ is the sum over all possibilities to grow the trees $t_{1},\ldots,t_{n}$ out of vertices of $t$ as the left-most child, with the extra condition that if two or more trees $t_{1},\ldots,t_{k}$, $k\geq 2$ grow out of the same vertex $v$ in $t$, then they need to have the same order as children of $t$ as they have as trees in $t_1\cdots t_n$.
\begin{exam}
\[
\tikz[baseline=(current bounding box.center), scale=0.7]{
  \node[circle,fill=black,inner sep=1.5pt] (a) at (0,0) {};
  \node[below=1pt, font=\footnotesize] at (a) {$a$};
  \node[circle,fill=black,inner sep=1.5pt] (b) at (0,0.6) {};
  \node[above=1pt, font=\footnotesize] at (b) {$b$};
  \draw (a)--(b);
}
\;\curvearrowrightl\;
\tikz[baseline=(current bounding box.center), scale=0.7]{
  \node[circle,fill=black,inner sep=1.5pt] (e) at (0,0) {};
  \node[below=1pt, font=\footnotesize] at (e) {$e$};
  \node[circle,fill=black,inner sep=1.5pt] (c) at (-0.5,0.6) {};
  \node[above left=1pt, font=\footnotesize] at (c) {$c$};
  \node[circle,fill=black,inner sep=1.5pt] (d) at (0.5,0.6) {};
  \node[above right=1pt, font=\footnotesize] at (d) {$d$};
  \draw (e)--(c) (e)--(d);
}
\;=\;
\tikz[baseline=(current bounding box.center), scale=0.7]{
  \node[circle,fill=black,inner sep=1.5pt] (e) at (0,0) {};
  \node[below=1pt, font=\footnotesize] at (e) {$e$};
  \node[circle,fill=black,inner sep=1.5pt] (c) at (-0.5,0.6) {};
  \node[left=1pt, font=\footnotesize] at (c) {$c$};
  \node[circle,fill=black,inner sep=1.5pt] (d) at (0.5,0.6) {};
  \node[right=1pt, font=\footnotesize] at (d) {$d$};
  \node[circle,fill=black,inner sep=1.5pt] (a) at (-0.5,1.2) {};
  \node[right=1pt, font=\footnotesize] at (a) {$a$};
  \node[circle,fill=black,inner sep=1.5pt] (b) at (-0.5,1.8) {};
  \node[left=1pt, font=\footnotesize] at (b) {$b$};
  \draw (e)--(c) (e)--(d) (c)--(a) (a)--(b);
}
\;+\;
\tikz[baseline=(current bounding box.center), scale=0.7]{
  \node[circle,fill=black,inner sep=1.5pt] (e) at (0,0) {};
  \node[below=1pt, font=\footnotesize] at (e) {$e$};
  \node[circle,fill=black,inner sep=1.5pt] (c) at (-0.7,0.6) {};
  \node[above=1pt, font=\footnotesize] at (c) {$c$};
  \node[circle,fill=black,inner sep=1.5pt] (a) at (0,0.6) {};
  \node[right=1pt, font=\footnotesize] at (a) {$a$};
  \node[circle,fill=black,inner sep=1.5pt] (d) at (0.7,0.6) {};
  \node[above=1pt, font=\footnotesize] at (d) {$d$};
  \node[circle,fill=black,inner sep=1.5pt] (b) at (0,1.2) {};
  \node[above=1pt, font=\footnotesize] at (b) {$b$};
  \draw (e)--(c) (e)--(a) (e)--(d) (a)--(b);
}
\;+\;
\tikz[baseline=(current bounding box.center), scale=0.7]{
  \node[circle,fill=black,inner sep=1.5pt] (e) at (0,0) {};
  \node[below=1pt, font=\footnotesize] at (e) {$e$};
  \node[circle,fill=black,inner sep=1.5pt] (c) at (-0.5,0.6) {};
  \node[left=1pt, font=\footnotesize] at (c) {$c$};
  \node[circle,fill=black,inner sep=1.5pt] (d) at (0.5,0.6) {};
  \node[right=1pt, font=\footnotesize] at (d) {$d$};
  \node[circle,fill=black,inner sep=1.5pt] (a) at (0.5,1.2) {};
  \node[left=1pt, font=\footnotesize] at (a) {$a$};
  \node[circle,fill=black,inner sep=1.5pt] (b) at (0.5,1.8) {};
  \node[right=1pt, font=\footnotesize] at (b) {$b$};
  \draw (e)--(c) (e)--(d) (d)--(a) (a)--(b);
}
\,,
\]
\[
\Bigl(
\TIKZ{
  \begin{scope}[scale=0.77]
    \node[dot] (a1) at (0,0) {};
    \node[right=1pt] at (a1) {$a$};
    \node[dot] (b1) at (0.65,0) {};
    \node[right=1pt] at (b1) {$b$};
  \end{scope}
}
\Bigr)
\;\curvearrowright_{l}\;
\TIKZ{
  \begin{scope}[scale=0.77]
    \node[dot] (c) at (0,0) {};
    \node[right=1pt] at (c) {$c$};
    \node[dot] (d) at (0,0.95) {};
    \node[right=1pt] at (d) {$d$};
    \draw (c)--(d);
  \end{scope}
}
=
\TIKZ{
  \begin{scope}[scale=0.77]
    \node[dot] (c) at (0,0) {};
    \node[right=1pt] at (c) {$c$};
    \node[dot] (d) at (0,0.95) {};
    \node[right=1pt] at (d) {$d$};
    \node[dot] (a) at (-0.70,1.75) {};
    \node[right=1pt] at (a) {$a$};
    \node[dot] (b) at ( 0.70,1.75) {};
    \node[right=1pt] at (b) {$b$};
    \draw (c)--(d) (d)--(a) (d)--(b);
  \end{scope}
}
+
\TIKZ{
  \begin{scope}[scale=0.77]
    \node[dot] (c) at (0,0) {};
    \node[right=1pt] at (c) {$c$};
    \node[dot] (a) at (-0.85,0.95) {};
    \node[right=1pt] at (a) {$a$};
    \node[dot] (d) at ( 0.55,0.95) {};
    \node[right=1pt] at (d) {$d$};
    \node[dot] (b) at ( 0.55,1.75) {};
    \node[right=1pt] at (b) {$b$};
    \draw (c)--(a) (c)--(d) (d)--(b);
  \end{scope}
}
+
\TIKZ{
  \begin{scope}[scale=0.77]
    \node[dot] (c) at (0,0) {};
    \node[right=1pt] at (c) {$c$};
    \node[dot] (b) at (-0.85,0.95) {};
    \node[right=1pt] at (b) {$b$};
    \node[dot] (d) at ( 0.55,0.95) {};
    \node[right=1pt] at (d) {$d$};
    \node[dot] (a) at ( 0.55,1.75) {};
    \node[right=1pt] at (a) {$a$};
    \draw (c)--(b) (c)--(d) (d)--(a);
  \end{scope}
}
+
\TIKZ{
  \begin{scope}[scale=0.77]
    \node[dot] (c) at (0,0) {};
    \node[right=1pt] at (c) {$c$};
    \node[dot] (a) at (-0.95,0.95) {};
    \node[right=1pt] at (a) {$a$};
    \node[dot] (b) at ( 0.00,0.95) {};
    \node[right=1pt] at (b) {$b$};
    \node[dot] (d) at ( 0.95,0.95) {};
    \node[right=1pt] at (d) {$d$};
    \draw (c)--(a) (c)--(b) (c)--(d);
  \end{scope}
}
\,.
\]
\end{exam}

The primitive elements of the graded dual Hopf algebra
$(\mathcal{H}_{\mathrm{MKW}}^A)^{g}$ form the free post-Lie algebra~\cite{KL2023,MKW2008}
\[
\bigl({\rm PSL}(A),\ \curvearrowright_{l},\ \lpp{-}{-}\bigr)
\]
on the set $A$.
Since $(\mathcal{H}_{\mathrm{MKW}}^A)^{g}$ as a coalgebra is the shuffle coalgebra
$\bigl(T(\bfk \OT^A),\Delta_{\shuffle}\bigr)$,
the primitive elements of the shuffle coalgebra $T(\OT^A)$  are the elements of the free Lie algebra ${\rm Lie}(\OT^A)$ generated by the vector space $\bfk\OT^A$. We next define the linear map
\begin{equation}
\Phi_{\mathrm{fer}}^{\mathcal T}: \bfk \OT^A \longrightarrow V(A),
\quad
t\longmapsto z^{d(v_1)}_{f(v_1)-1}\cdots z^{d(v_p)}_{f(v_p)-1},
\mlabel{eq:treelev}
\end{equation}
called {\bf planar tree fertility map}.
Here $v_1<\cdots<v_p$ are the vertices of $t$, ordered  by the root-subtree order in Section~\ref{sec:order},
$d(v_i)\in A$ is the decoration of $v_i$, and $f(v_i)$ denotes the fertility of $v_i$ (i.e. the number of the incoming edges).\\

The next proposition establishes the fundamental correspondence between $A$-decorated planar rooted trees and monomials  $\bzia\in V(A)$. Let $\widetilde{\OF}^{A}$ denote the $\bfk$-linear span of
$A$-decorated planar rooted forests with ordered free edges.  A free edge is
an ordered open child-slot attached to a vertex but not yet attached to
another vertex.  For a vertex in such a forest, ordinary child edges and free
edges are counted together when computing its number of child slots.  The
ordering of the free edges is part of the planar structure and records the
left-to-right order in which later vertices may be attached.

\begin{prop}\mlabel{pro:Phibi}
The  planar tree fertility map in \eqref{eq:treelev} is  a linear isomorphism onto $V(A)$.
\end{prop}

\begin{proof}
We construct the inverse map on monomials.  For $a\in A$ and $i\ge -1$,
let $C_i^a$ be the one-vertex corolla decorated by $a$, endowed with
$i+1$ ordered free edges. Define a linear map
\begin{equation}\mlabel{eq:psi}
\Psi:\bnca \longrightarrow \widetilde{\OF}^{A},
\quad
z_{i_1}^{a_1}\cdots z_{i_p}^{a_p}\longmapsto \Psi(w),
\end{equation}
where $\Psi(w)$ is constructed recursively on the length $p=|w|$ as
follows.  If $p=1$, set $\Psi(z_{i_1}^{a_1})=C_{i_1}^{a_1}$.  In particular, if
$z_{i_1}^{a_1}\in V(A)$, then $i_1=-1$, and
$C_{-1}^{a_1}$ has no free edge.  Hence $\Psi(z_{-1}^{a_1})$ is the ordinary
one-vertex tree decorated by $a_1$.
For $p>1$, start with $C_{i_1}^{a_1}$.  Suppose that the forest associated
with the prefix
$z_{i_1}^{a_1}\cdots z_{i_{j-1}}^{a_{j-1}}$
has already been constructed.  To insert the $j$-th letter, attach the root
of $C_{i_j}^{a_j}$ to the leftmost available free edge.  If no free edge is
available, put $C_{i_j}^{a_j}$ as a new connected component on the right.
The used free edges are then replaced by the ordered free edges of
$C_{i_j}^{a_j}$.  After $p$ steps, this gives
$\Psi\bigl(z_{i_1}^{a_1}\cdots z_{i_p}^{a_p}\bigr).$\\

Now let $w=z_{i_1}^{a_1}\cdots z_{i_p}^{a_p}\in V(A)$.  After
the first $q$ letters have been read, the number of available free edges is
$1+\sum_{r=1}^{q} i_r .$
Since $w\in V(A)$, the prefix condition implies that this number
is positive for $1\le q<p$.  Hence no new connected component is created after the first
one.  Since $w$ has total weight $-1$, after the last letter the number
of free edges is
$1+\sum_{r=1}^{p} i_r=0.$
Therefore $\Psi(w)$ is an ordinary connected $A$-decorated planar rooted
tree. It remains to check that $\Psi$ is inverse to
$\Phi_{\mathrm{fer}}^{\mathcal T}$.  By construction,
\[
\Phi_{\mathrm{fer}}^{\mathcal T}\big(\Psi(w)\big)=w,
\quad
w\in V(A),
\]
because the vertices of $\Psi(w)$ are read in the root-subtree order, and a
vertex decorated by $a$ with $i+1$ child slots yields the letter
$z_i^a$.  Conversely, if $t\in\OT^A$, then applying $\Psi$ to
$\Phi_{\mathrm{fer}}^{\mathcal T}(t)$ reconstructs $t$, since the
recursive construction fills the free edges exactly in the root-subtree order.
Thus
\[
\Phi_{\mathrm{fer}}^{\mathcal T}\circ \Psi=\id_{V(A)},
\quad
\Psi\circ \Phi_{\mathrm{fer}}^{\mathcal T}=\id_{\OT^A}.
\]
Hence the planar tree fertility map is a linear isomorphism.
\end{proof}

The restriction of $\Phi_{\rm fer}$ to trees without free edges is exactly
$\Phi_{\mathrm{fer}}^{\mathcal T}$.  
We illustrate the inverse construction $\Psi$ in \meqref{eq:psi} with the following example.

\begin{exam} Consider
\[
\bzia
=
z_{1}^{a_{1}} z_{-1}^{a_{2}} z_{2}^{a_{3}} z_{0}^{a_{4}}
z_{-1}^{a_{5}} z_{-1}^{a_{6}} z_{-1}^{a_{7}}.
\]
Then the process of obtaining $\Psi(\bzia)$ is as follows:

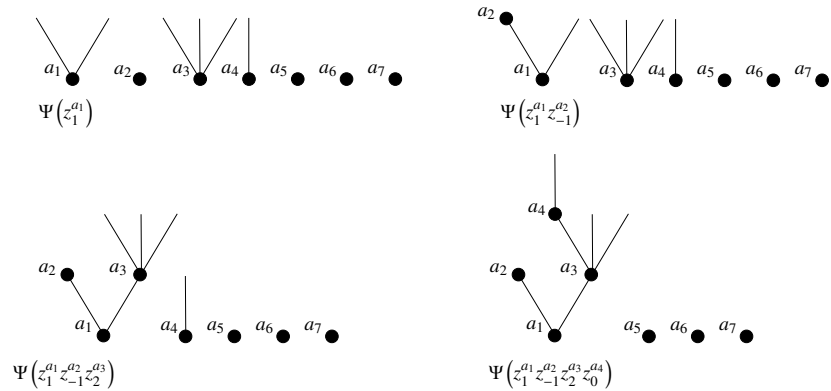
\begingroup
\newcommand{\PsiConstructionPicture}{%
\draw  [fill={rgb, 255:red, 0; green, 0; blue, 0 }  ,fill opacity=1 ] (360.33,483.83) .. controls (360.33,480.98) and (362.65,478.67) .. (365.5,478.67) .. controls (368.35,478.67) and (370.67,480.98) .. (370.67,483.83) .. controls (370.67,486.69) and (368.35,489) .. (365.5,489) .. controls (362.65,489) and (360.33,486.69) .. (360.33,483.83) -- cycle ;
\draw  [fill={rgb, 255:red, 0; green, 0; blue, 0 }  ,fill opacity=1 ] (390.33,433.83) .. controls (390.33,430.98) and (392.65,428.67) .. (395.5,428.67) .. controls (398.35,428.67) and (400.67,430.98) .. (400.67,433.83) .. controls (400.67,436.69) and (398.35,439) .. (395.5,439) .. controls (392.65,439) and (390.33,436.69) .. (390.33,433.83) -- cycle ;
\draw  [fill={rgb, 255:red, 0; green, 0; blue, 0 }  ,fill opacity=1 ] (330.67,433.83) .. controls (330.67,430.98) and (332.98,428.67) .. (335.83,428.67) .. controls (338.69,428.67) and (341,430.98) .. (341,433.83) .. controls (341,436.69) and (338.69,439) .. (335.83,439) .. controls (332.98,439) and (330.67,436.69) .. (330.67,433.83) -- cycle ;
\draw  [fill={rgb, 255:red, 0; green, 0; blue, 0 }  ,fill opacity=1 ] (360.67,384.17) .. controls (360.67,381.31) and (362.98,379) .. (365.83,379) .. controls (368.69,379) and (371,381.31) .. (371,384.17) .. controls (371,387.02) and (368.69,389.33) .. (365.83,389.33) .. controls (362.98,389.33) and (360.67,387.02) .. (360.67,384.17) -- cycle ;
\draw  [fill={rgb, 255:red, 0; green, 0; blue, 0 }  ,fill opacity=1 ] (390.67,384.17) .. controls (390.67,381.31) and (392.98,379) .. (395.83,379) .. controls (398.69,379) and (401,381.31) .. (401,384.17) .. controls (401,387.02) and (398.69,389.33) .. (395.83,389.33) .. controls (392.98,389.33) and (390.67,387.02) .. (390.67,384.17) -- cycle ;
\draw  [fill={rgb, 255:red, 0; green, 0; blue, 0 }  ,fill opacity=1 ] (459.67,484.83) .. controls (459.67,481.98) and (461.98,479.67) .. (464.83,479.67) .. controls (467.69,479.67) and (470,481.98) .. (470,484.83) .. controls (470,487.69) and (467.69,490) .. (464.83,490) .. controls (461.98,490) and (459.67,487.69) .. (459.67,484.83) -- cycle ;
\draw  [fill={rgb, 255:red, 0; green, 0; blue, 0 }  ,fill opacity=1 ] (360.67,334.17) .. controls (360.67,331.31) and (362.98,329) .. (365.83,329) .. controls (368.69,329) and (371,331.31) .. (371,334.17) .. controls (371,337.02) and (368.69,339.33) .. (365.83,339.33) .. controls (362.98,339.33) and (360.67,337.02) .. (360.67,334.17) -- cycle ;
\draw    (335.67,433.67) -- (366,484) ;
\draw    (366.33,384.33) -- (396.67,434.67) ;
\draw    (396.33,384.33) -- (396.67,434) ;
\draw    (365.67,335) -- (366,384.67) ;
\draw    (426,384) -- (396,434) ;
\draw    (396,434) -- (366,484) ;
\draw  [fill={rgb, 255:red, 0; green, 0; blue, 0 }  ,fill opacity=1 ] (619.67,484.83) .. controls (619.67,481.98) and (621.98,479.67) .. (624.83,479.67) .. controls (627.69,479.67) and (630,481.98) .. (630,484.83) .. controls (630,487.69) and (627.69,490) .. (624.83,490) .. controls (621.98,490) and (619.67,487.69) .. (619.67,484.83) -- cycle ;
\draw  [fill={rgb, 255:red, 0; green, 0; blue, 0 }  ,fill opacity=1 ] (649.67,434.83) .. controls (649.67,431.98) and (651.98,429.67) .. (654.83,429.67) .. controls (657.69,429.67) and (660,431.98) .. (660,434.83) .. controls (660,437.69) and (657.69,440) .. (654.83,440) .. controls (651.98,440) and (649.67,437.69) .. (649.67,434.83) -- cycle ;
\draw  [fill={rgb, 255:red, 0; green, 0; blue, 0 }  ,fill opacity=1 ] (590,434.83) .. controls (590,431.98) and (592.31,429.67) .. (595.17,429.67) .. controls (598.02,429.67) and (600.33,431.98) .. (600.33,434.83) .. controls (600.33,437.69) and (598.02,440) .. (595.17,440) .. controls (592.31,440) and (590,437.69) .. (590,434.83) -- cycle ;
\draw  [fill={rgb, 255:red, 0; green, 0; blue, 0 }  ,fill opacity=1 ] (620,385.17) .. controls (620,382.31) and (622.31,380) .. (625.17,380) .. controls (628.02,380) and (630.33,382.31) .. (630.33,385.17) .. controls (630.33,388.02) and (628.02,390.33) .. (625.17,390.33) .. controls (622.31,390.33) and (620,388.02) .. (620,385.17) -- cycle ;
\draw  [fill={rgb, 255:red, 0; green, 0; blue, 0 }  ,fill opacity=1 ] (650,385.17) .. controls (650,382.31) and (652.31,380) .. (655.17,380) .. controls (658.02,380) and (660.33,382.31) .. (660.33,385.17) .. controls (660.33,388.02) and (658.02,390.33) .. (655.17,390.33) .. controls (652.31,390.33) and (650,388.02) .. (650,385.17) -- cycle ;
\draw  [fill={rgb, 255:red, 0; green, 0; blue, 0 }  ,fill opacity=1 ] (679.67,385.17) .. controls (679.67,382.31) and (681.98,380) .. (684.83,380) .. controls (687.69,380) and (690,382.31) .. (690,385.17) .. controls (690,388.02) and (687.69,390.33) .. (684.83,390.33) .. controls (681.98,390.33) and (679.67,388.02) .. (679.67,385.17) -- cycle ;
\draw  [fill={rgb, 255:red, 0; green, 0; blue, 0 }  ,fill opacity=1 ] (620,335.17) .. controls (620,332.31) and (622.31,330) .. (625.17,330) .. controls (628.02,330) and (630.33,332.31) .. (630.33,335.17) .. controls (630.33,338.02) and (628.02,340.33) .. (625.17,340.33) .. controls (622.31,340.33) and (620,338.02) .. (620,335.17) -- cycle ;
\draw    (595,434.67) -- (625.33,485) ;
\draw    (625.67,385.33) -- (656,435.67) ;
\draw    (655.67,385.33) -- (656,435) ;
\draw    (625,336) -- (625.33,385.67) ;
\draw    (685.33,385) -- (655.33,435) ;
\draw    (655.33,435) -- (625.33,485) ;
\draw  [fill={rgb, 255:red, 0; green, 0; blue, 0 }  ,fill opacity=1 ] (110.33,484.17) .. controls (110.33,481.31) and (112.65,479) .. (115.5,479) .. controls (118.35,479) and (120.67,481.31) .. (120.67,484.17) .. controls (120.67,487.02) and (118.35,489.33) .. (115.5,489.33) .. controls (112.65,489.33) and (110.33,487.02) .. (110.33,484.17) -- cycle ;
\draw  [fill={rgb, 255:red, 0; green, 0; blue, 0 }  ,fill opacity=1 ] (140.33,434.17) .. controls (140.33,431.31) and (142.65,429) .. (145.5,429) .. controls (148.35,429) and (150.67,431.31) .. (150.67,434.17) .. controls (150.67,437.02) and (148.35,439.33) .. (145.5,439.33) .. controls (142.65,439.33) and (140.33,437.02) .. (140.33,434.17) -- cycle ;
\draw  [fill={rgb, 255:red, 0; green, 0; blue, 0 }  ,fill opacity=1 ] (80.67,434.17) .. controls (80.67,431.31) and (82.98,429) .. (85.83,429) .. controls (88.69,429) and (91,431.31) .. (91,434.17) .. controls (91,437.02) and (88.69,439.33) .. (85.83,439.33) .. controls (82.98,439.33) and (80.67,437.02) .. (80.67,434.17) -- cycle ;
\draw  [fill={rgb, 255:red, 0; green, 0; blue, 0 }  ,fill opacity=1 ] (110.67,384.5) .. controls (110.67,381.65) and (112.98,379.33) .. (115.83,379.33) .. controls (118.69,379.33) and (121,381.65) .. (121,384.5) .. controls (121,387.35) and (118.69,389.67) .. (115.83,389.67) .. controls (112.98,389.67) and (110.67,387.35) .. (110.67,384.5) -- cycle ;
\draw  [fill={rgb, 255:red, 0; green, 0; blue, 0 }  ,fill opacity=1 ] (200,485.83) .. controls (200,482.98) and (202.31,480.67) .. (205.17,480.67) .. controls (208.02,480.67) and (210.33,482.98) .. (210.33,485.83) .. controls (210.33,488.69) and (208.02,491) .. (205.17,491) .. controls (202.31,491) and (200,488.69) .. (200,485.83) -- cycle ;
\draw  [fill={rgb, 255:red, 0; green, 0; blue, 0 }  ,fill opacity=1 ] (239.67,485.83) .. controls (239.67,482.98) and (241.98,480.67) .. (244.83,480.67) .. controls (247.69,480.67) and (250,482.98) .. (250,485.83) .. controls (250,488.69) and (247.69,491) .. (244.83,491) .. controls (241.98,491) and (239.67,488.69) .. (239.67,485.83) -- cycle ;
\draw  [fill={rgb, 255:red, 0; green, 0; blue, 0 }  ,fill opacity=1 ] (110.67,334.5) .. controls (110.67,331.65) and (112.98,329.33) .. (115.83,329.33) .. controls (118.69,329.33) and (121,331.65) .. (121,334.5) .. controls (121,337.35) and (118.69,339.67) .. (115.83,339.67) .. controls (112.98,339.67) and (110.67,337.35) .. (110.67,334.5) -- cycle ;
\draw    (85.67,434) -- (116,484.33) ;
\draw    (116.33,384.67) -- (146.67,435) ;
\draw    (146.33,384.67) -- (146.67,434.33) ;
\draw    (115.67,335.33) -- (116,385) ;
\draw    (176,384.33) -- (146,434.33) ;
\draw    (146,434.33) -- (116,484.33) ;
\draw  [fill={rgb, 255:red, 0; green, 0; blue, 0 }  ,fill opacity=1 ] (450.33,274.17) .. controls (450.33,271.31) and (452.65,269) .. (455.5,269) .. controls (458.35,269) and (460.67,271.31) .. (460.67,274.17) .. controls (460.67,277.02) and (458.35,279.33) .. (455.5,279.33) .. controls (452.65,279.33) and (450.33,277.02) .. (450.33,274.17) -- cycle ;
\draw  [fill={rgb, 255:red, 0; green, 0; blue, 0 }  ,fill opacity=1 ] (480.33,224.17) .. controls (480.33,221.31) and (482.65,219) .. (485.5,219) .. controls (488.35,219) and (490.67,221.31) .. (490.67,224.17) .. controls (490.67,227.02) and (488.35,229.33) .. (485.5,229.33) .. controls (482.65,229.33) and (480.33,227.02) .. (480.33,224.17) -- cycle ;
\draw  [fill={rgb, 255:red, 0; green, 0; blue, 0 }  ,fill opacity=1 ] (420.67,224.17) .. controls (420.67,221.31) and (422.98,219) .. (425.83,219) .. controls (428.69,219) and (431,221.31) .. (431,224.17) .. controls (431,227.02) and (428.69,229.33) .. (425.83,229.33) .. controls (422.98,229.33) and (420.67,227.02) .. (420.67,224.17) -- cycle ;
\draw  [fill={rgb, 255:red, 0; green, 0; blue, 0 }  ,fill opacity=1 ] (450.67,174.5) .. controls (450.67,171.65) and (452.98,169.33) .. (455.83,169.33) .. controls (458.69,169.33) and (461,171.65) .. (461,174.5) .. controls (461,177.35) and (458.69,179.67) .. (455.83,179.67) .. controls (452.98,179.67) and (450.67,177.35) .. (450.67,174.5) -- cycle ;
\draw  [fill={rgb, 255:red, 0; green, 0; blue, 0 }  ,fill opacity=1 ] (567.67,274.83) .. controls (567.67,271.98) and (569.98,269.67) .. (572.83,269.67) .. controls (575.69,269.67) and (578,271.98) .. (578,274.83) .. controls (578,277.69) and (575.69,280) .. (572.83,280) .. controls (569.98,280) and (567.67,277.69) .. (567.67,274.83) -- cycle ;
\draw  [fill={rgb, 255:red, 0; green, 0; blue, 0 }  ,fill opacity=1 ] (607.67,274.83) .. controls (607.67,271.98) and (609.98,269.67) .. (612.83,269.67) .. controls (615.69,269.67) and (618,271.98) .. (618,274.83) .. controls (618,277.69) and (615.69,280) .. (612.83,280) .. controls (609.98,280) and (607.67,277.69) .. (607.67,274.83) -- cycle ;
\draw  [fill={rgb, 255:red, 0; green, 0; blue, 0 }  ,fill opacity=1 ] (527.67,274.83) .. controls (527.67,271.98) and (529.98,269.67) .. (532.83,269.67) .. controls (535.69,269.67) and (538,271.98) .. (538,274.83) .. controls (538,277.69) and (535.69,280) .. (532.83,280) .. controls (529.98,280) and (527.67,277.69) .. (527.67,274.83) -- cycle ;
\draw    (425.67,224) -- (456,274.33) ;
\draw    (456.33,174.67) -- (486.67,225) ;
\draw    (486.33,174.67) -- (486.67,224.33) ;
\draw    (455.67,125.33) -- (456,175) ;
\draw    (516,174.33) -- (486,224.33) ;
\draw    (486,224.33) -- (456,274.33) ;
\draw  [fill={rgb, 255:red, 0; green, 0; blue, 0 }  ,fill opacity=1 ] (80.33,274.17) .. controls (80.33,271.31) and (82.65,269) .. (85.5,269) .. controls (88.35,269) and (90.67,271.31) .. (90.67,274.17) .. controls (90.67,277.02) and (88.35,279.33) .. (85.5,279.33) .. controls (82.65,279.33) and (80.33,277.02) .. (80.33,274.17) -- cycle ;
\draw  [fill={rgb, 255:red, 0; green, 0; blue, 0 }  ,fill opacity=1 ] (110.33,224.17) .. controls (110.33,221.31) and (112.65,219) .. (115.5,219) .. controls (118.35,219) and (120.67,221.31) .. (120.67,224.17) .. controls (120.67,227.02) and (118.35,229.33) .. (115.5,229.33) .. controls (112.65,229.33) and (110.33,227.02) .. (110.33,224.17) -- cycle ;
\draw  [fill={rgb, 255:red, 0; green, 0; blue, 0 }  ,fill opacity=1 ] (50.67,224.17) .. controls (50.67,221.31) and (52.98,219) .. (55.83,219) .. controls (58.69,219) and (61,221.31) .. (61,224.17) .. controls (61,227.02) and (58.69,229.33) .. (55.83,229.33) .. controls (52.98,229.33) and (50.67,227.02) .. (50.67,224.17) -- cycle ;
\draw  [fill={rgb, 255:red, 0; green, 0; blue, 0 }  ,fill opacity=1 ] (147.67,274.83) .. controls (147.67,271.98) and (149.98,269.67) .. (152.83,269.67) .. controls (155.69,269.67) and (158,271.98) .. (158,274.83) .. controls (158,277.69) and (155.69,280) .. (152.83,280) .. controls (149.98,280) and (147.67,277.69) .. (147.67,274.83) -- cycle ;
\draw  [fill={rgb, 255:red, 0; green, 0; blue, 0 }  ,fill opacity=1 ] (227.67,274.83) .. controls (227.67,271.98) and (229.98,269.67) .. (232.83,269.67) .. controls (235.69,269.67) and (238,271.98) .. (238,274.83) .. controls (238,277.69) and (235.69,280) .. (232.83,280) .. controls (229.98,280) and (227.67,277.69) .. (227.67,274.83) -- cycle ;
\draw  [fill={rgb, 255:red, 0; green, 0; blue, 0 }  ,fill opacity=1 ] (267.67,274.83) .. controls (267.67,271.98) and (269.98,269.67) .. (272.83,269.67) .. controls (275.69,269.67) and (278,271.98) .. (278,274.83) .. controls (278,277.69) and (275.69,280) .. (272.83,280) .. controls (269.98,280) and (267.67,277.69) .. (267.67,274.83) -- cycle ;
\draw  [fill={rgb, 255:red, 0; green, 0; blue, 0 }  ,fill opacity=1 ] (187.67,274.83) .. controls (187.67,271.98) and (189.98,269.67) .. (192.83,269.67) .. controls (195.69,269.67) and (198,271.98) .. (198,274.83) .. controls (198,277.69) and (195.69,280) .. (192.83,280) .. controls (189.98,280) and (187.67,277.69) .. (187.67,274.83) -- cycle ;
\draw    (152.67,225.33) -- (153,275) ;
\draw    (55.67,224) -- (86,274.33) ;
\draw    (86.33,174.67) -- (116.67,225) ;
\draw    (116.33,174.67) -- (116.67,224.33) ;
\draw    (146,174.33) -- (116,224.33) ;
\draw    (116,224.33) -- (86,274.33) ;
\draw  [fill={rgb, 255:red, 0; green, 0; blue, 0 }  ,fill opacity=1 ] (440.33,64.17) .. controls (440.33,61.31) and (442.65,59) .. (445.5,59) .. controls (448.35,59) and (450.67,61.31) .. (450.67,64.17) .. controls (450.67,67.02) and (448.35,69.33) .. (445.5,69.33) .. controls (442.65,69.33) and (440.33,67.02) .. (440.33,64.17) -- cycle ;
\draw  [fill={rgb, 255:red, 0; green, 0; blue, 0 }  ,fill opacity=1 ] (509.67,64.83) .. controls (509.67,61.98) and (511.98,59.67) .. (514.83,59.67) .. controls (517.69,59.67) and (520,61.98) .. (520,64.83) .. controls (520,67.69) and (517.69,70) .. (514.83,70) .. controls (511.98,70) and (509.67,67.69) .. (509.67,64.83) -- cycle ;
\draw  [fill={rgb, 255:red, 0; green, 0; blue, 0 }  ,fill opacity=1 ] (410.67,14.17) .. controls (410.67,11.31) and (412.98,9) .. (415.83,9) .. controls (418.69,9) and (421,11.31) .. (421,14.17) .. controls (421,17.02) and (418.69,19.33) .. (415.83,19.33) .. controls (412.98,19.33) and (410.67,17.02) .. (410.67,14.17) -- cycle ;
\draw  [fill={rgb, 255:red, 0; green, 0; blue, 0 }  ,fill opacity=1 ] (549.67,64.83) .. controls (549.67,61.98) and (551.98,59.67) .. (554.83,59.67) .. controls (557.69,59.67) and (560,61.98) .. (560,64.83) .. controls (560,67.69) and (557.69,70) .. (554.83,70) .. controls (551.98,70) and (549.67,67.69) .. (549.67,64.83) -- cycle ;
\draw  [fill={rgb, 255:red, 0; green, 0; blue, 0 }  ,fill opacity=1 ] (629.67,64.83) .. controls (629.67,61.98) and (631.98,59.67) .. (634.83,59.67) .. controls (637.69,59.67) and (640,61.98) .. (640,64.83) .. controls (640,67.69) and (637.69,70) .. (634.83,70) .. controls (631.98,70) and (629.67,67.69) .. (629.67,64.83) -- cycle ;
\draw  [fill={rgb, 255:red, 0; green, 0; blue, 0 }  ,fill opacity=1 ] (669.67,64.83) .. controls (669.67,61.98) and (671.98,59.67) .. (674.83,59.67) .. controls (677.69,59.67) and (680,61.98) .. (680,64.83) .. controls (680,67.69) and (677.69,70) .. (674.83,70) .. controls (671.98,70) and (669.67,67.69) .. (669.67,64.83) -- cycle ;
\draw  [fill={rgb, 255:red, 0; green, 0; blue, 0 }  ,fill opacity=1 ] (589.67,64.83) .. controls (589.67,61.98) and (591.98,59.67) .. (594.83,59.67) .. controls (597.69,59.67) and (600,61.98) .. (600,64.83) .. controls (600,67.69) and (597.69,70) .. (594.83,70) .. controls (591.98,70) and (589.67,67.69) .. (589.67,64.83) -- cycle ;
\draw    (415.67,14) -- (446,64.33) ;
\draw    (484.33,14.67) -- (514.67,65) ;
\draw    (514.33,15.33) -- (514.67,65) ;
\draw    (554.67,14.67) -- (554.67,64.67) ;
\draw    (544.67,15) -- (514.67,65) ;
\draw    (475.33,14.33) -- (445.33,64.33) ;
\draw  [fill={rgb, 255:red, 0; green, 0; blue, 0 }  ,fill opacity=1 ] (55,63.83) .. controls (55,60.98) and (57.31,58.67) .. (60.17,58.67) .. controls (63.02,58.67) and (65.33,60.98) .. (65.33,63.83) .. controls (65.33,66.69) and (63.02,69) .. (60.17,69) .. controls (57.31,69) and (55,66.69) .. (55,63.83) -- cycle ;
\draw    (30.33,13.67) -- (60.67,64) ;
\draw    (90,14) -- (60,64) ;
\draw  [fill={rgb, 255:red, 0; green, 0; blue, 0 }  ,fill opacity=1 ] (159.67,63.83) .. controls (159.67,60.98) and (161.98,58.67) .. (164.83,58.67) .. controls (167.69,58.67) and (170,60.98) .. (170,63.83) .. controls (170,66.69) and (167.69,69) .. (164.83,69) .. controls (161.98,69) and (159.67,66.69) .. (159.67,63.83) -- cycle ;
\draw  [fill={rgb, 255:red, 0; green, 0; blue, 0 }  ,fill opacity=1 ] (199.67,63.83) .. controls (199.67,60.98) and (201.98,58.67) .. (204.83,58.67) .. controls (207.69,58.67) and (210,60.98) .. (210,63.83) .. controls (210,66.69) and (207.69,69) .. (204.83,69) .. controls (201.98,69) and (199.67,66.69) .. (199.67,63.83) -- cycle ;
\draw  [fill={rgb, 255:red, 0; green, 0; blue, 0 }  ,fill opacity=1 ] (279.67,63.83) .. controls (279.67,60.98) and (281.98,58.67) .. (284.83,58.67) .. controls (287.69,58.67) and (290,60.98) .. (290,63.83) .. controls (290,66.69) and (287.69,69) .. (284.83,69) .. controls (281.98,69) and (279.67,66.69) .. (279.67,63.83) -- cycle ;
\draw  [fill={rgb, 255:red, 0; green, 0; blue, 0 }  ,fill opacity=1 ] (319.67,63.83) .. controls (319.67,60.98) and (321.98,58.67) .. (324.83,58.67) .. controls (327.69,58.67) and (330,60.98) .. (330,63.83) .. controls (330,66.69) and (327.69,69) .. (324.83,69) .. controls (321.98,69) and (319.67,66.69) .. (319.67,63.83) -- cycle ;
\draw  [fill={rgb, 255:red, 0; green, 0; blue, 0 }  ,fill opacity=1 ] (239.67,63.83) .. controls (239.67,60.98) and (241.98,58.67) .. (244.83,58.67) .. controls (247.69,58.67) and (250,60.98) .. (250,63.83) .. controls (250,66.69) and (247.69,69) .. (244.83,69) .. controls (241.98,69) and (239.67,66.69) .. (239.67,63.83) -- cycle ;
\draw    (134.33,13.67) -- (164.67,64) ;
\draw    (164.33,14.33) -- (164.67,64) ;
\draw    (204.67,13.67) -- (204.67,63.67) ;
\draw    (194.67,14) -- (164.67,64) ;
\draw  [fill={rgb, 255:red, 0; green, 0; blue, 0 }  ,fill opacity=1 ] (110,63.83) .. controls (110,60.98) and (112.31,58.67) .. (115.17,58.67) .. controls (118.02,58.67) and (120.33,60.98) .. (120.33,63.83) .. controls (120.33,66.69) and (118.02,69) .. (115.17,69) .. controls (112.31,69) and (110,66.69) .. (110,63.83) -- cycle ;
\draw (668.33,360.73) node [anchor=north west][inner sep=0.75pt]    {$a_{7}$};
\draw (635.67,361.4) node [anchor=north west][inner sep=0.75pt]    {$a_{6}$};
\draw (600.33,323.07) node [anchor=north west][inner sep=0.75pt]    {$a_{5}$};
\draw (628.33,424.4) node [anchor=north west][inner sep=0.75pt]    {$a_{3}$};
\draw (570.33,424.4) node [anchor=north west][inner sep=0.75pt]    {$a_{2}$};
\draw (600.33,471.07) node [anchor=north west][inner sep=0.75pt]    {$a_{1}$};
\draw (601.67,375.07) node [anchor=north west][inner sep=0.75pt]    {$a_{4}$};
\draw (651,51.4) node [anchor=north west][inner sep=0.75pt]    {$a_{7}$};
\draw (611,51.4) node [anchor=north west][inner sep=0.75pt]    {$a_{6}$};
\draw (571,51.4) node [anchor=north west][inner sep=0.75pt]    {$a_{5}$};
\draw (491,51.4) node [anchor=north west][inner sep=0.75pt]    {$a_{3}$};
\draw (391,3.73) node [anchor=north west][inner sep=0.75pt]    {$a_{2}$};
\draw (421,50.4) node [anchor=north west][inner sep=0.75pt]    {$a_{1}$};
\draw (531,51.4) node [anchor=north west][inner sep=0.75pt]    {$a_{4}$};
\draw (35.67,50.07) node [anchor=north west][inner sep=0.75pt]    {$a_{1}$};
\draw (181,50.4) node [anchor=north west][inner sep=0.75pt]    {$a_{4}$};
\draw (141,50.4) node [anchor=north west][inner sep=0.75pt]    {$a_{3}$};
\draw (221,50.4) node [anchor=north west][inner sep=0.75pt]    {$a_{5}$};
\draw (261,50.4) node [anchor=north west][inner sep=0.75pt]    {$a_{6}$};
\draw (301,50.4) node [anchor=north west][inner sep=0.75pt]    {$a_{7}$};
\draw (93,51.4) node [anchor=north west][inner sep=0.75pt]    {$a_{2}$};
\draw (31,78.4) node [anchor=north west][inner sep=0.75pt]    {$\Psi \left( z_{1}^{a_{1}}\right)$};
\draw (410,78.4) node [anchor=north west][inner sep=0.75pt]    {$\Psi \left( z_{1}^{a_{1}} z_{-1}^{a_{2}}\right)$};
\draw (131,261.4) node [anchor=north west][inner sep=0.75pt]    {$a_{4}$};
\draw (169,261.4) node [anchor=north west][inner sep=0.75pt]    {$a_{5}$};
\draw (209,261.4) node [anchor=north west][inner sep=0.75pt]    {$a_{6}$};
\draw (249,261.4) node [anchor=north west][inner sep=0.75pt]    {$a_{7}$};
\draw (511,261.4) node [anchor=north west][inner sep=0.75pt]    {$a_{5}$};
\draw (549,261.4) node [anchor=north west][inner sep=0.75pt]    {$a_{6}$};
\draw (589,261.4) node [anchor=north west][inner sep=0.75pt]    {$a_{7}$};
\draw (10,292.4) node [anchor=north west][inner sep=0.75pt ]    {$\Psi \left( z_{1}^{a_{1}} z_{-1}^{a_{2}} z_{2}^{a_{3}}\right)$};
\draw (61,260.4) node [anchor=north west][inner sep=0.75pt]    {$a_{1}$};
\draw (31,213.73) node [anchor=north west][inner sep=0.75pt]    {$a_{2}$};
\draw (89,213.73) node [anchor=north west][inner sep=0.75pt]    {$a_{3}$};
\draw (41,500) node [anchor=north west][inner sep=0.75pt]    {$\Psi \left( z_{1}^{a_{1}} z_{-1}^{a_{2}} z_{2}^{a_{3}} z_{0}^{a_{4}} z_{-1}^{a_{5}}\right)$};
\draw (432.33,164.4) node [anchor=north west][inner sep=0.75pt]    {$a_{4}$};
\draw (431,260.4) node [anchor=north west][inner sep=0.75pt]    {$a_{1}$};
\draw (401,213.73) node [anchor=north west][inner sep=0.75pt]    {$a_{2}$};
\draw (459,213.73) node [anchor=north west][inner sep=0.75pt]    {$a_{3}$};
\draw (400,292.4) node [anchor=north west][inner sep=0.75pt ]    {$\Psi \left( z_{1}^{a_{1}} z_{-1}^{a_{2}} z_{2}^{a_{3}} z_{0}^{a_{4}}\right)$};
\draw (92.33,374.4) node [anchor=north west][inner sep=0.75pt]    {$a_{4}$};
\draw (91,470.4) node [anchor=north west][inner sep=0.75pt]    {$a_{1}$};
\draw (61,423.73) node [anchor=north west][inner sep=0.75pt]    {$a_{2}$};
\draw (119,423.73) node [anchor=north west][inner sep=0.75pt]    {$a_{3}$};
\draw (91,322.4) node [anchor=north west][inner sep=0.75pt]    {$a_{5}$};
\draw (181,471.4) node [anchor=north west][inner sep=0.75pt]    {$a_{6}$};
\draw (221,472.4) node [anchor=north west][inner sep=0.75pt]    {$a_{7}$};
\draw (342.33,374.07) node [anchor=north west][inner sep=0.75pt]    {$a_{4}$};
\draw (341,470.07) node [anchor=north west][inner sep=0.75pt]    {$a_{1}$};
\draw (311,423.4) node [anchor=north west][inner sep=0.75pt]    {$a_{2}$};
\draw (369,423.4) node [anchor=north west][inner sep=0.75pt]    {$a_{3}$};
\draw (341,322.07) node [anchor=north west][inner sep=0.75pt]    {$a_{5}$};
\draw (376.33,360.4) node [anchor=north west][inner sep=0.75pt]    {$a_{6}$};
\draw (441,471.4) node [anchor=north west][inner sep=0.75pt]    {$a_{7}$};
\draw (290,500) node [anchor=north west][inner sep=0.75pt ]    {$\Psi \left( z_{1}^{a_{1}} z_{-1}^{a_{2}} z_{2}^{a_{3}} z_{0}^{a_{4}} z_{-1}^{a_{5}} z_{-1}^{a_{6}}\right)$};
\draw (605,500) node [anchor=north west][inner sep=0.75pt ]    {$\Psi ( \bzia)$};
}

\[
\scalebox{0.68}{%
\begin{tikzpicture}[x=0.75pt,y=0.75pt,yscale=-0.9,xscale=0.9]
\path[use as bounding box] (-5,-10) rectangle (705,318);
\clip (-5,-10) rectangle (705,318);
\PsiConstructionPicture
\end{tikzpicture}%
}
\]

\[
\scalebox{0.68}
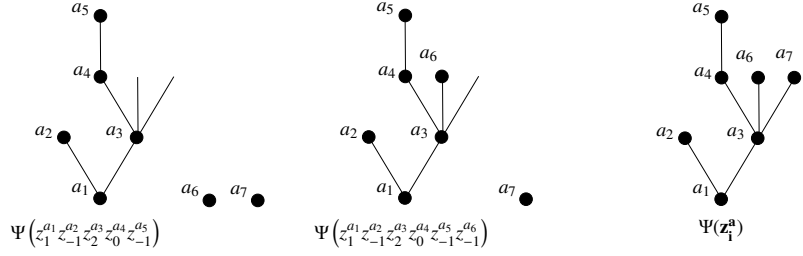
{%
\begin{tikzpicture}[x=0.75pt,y=0.75pt,yscale=-0.9,xscale=0.9]
\path[use as bounding box] (-5,318) rectangle (705,545);
\clip (-5,318) rectangle (705,545);
\PsiConstructionPicture
\end{tikzpicture}%
}
\]
\endgroup

\end{exam}

\begin{remark}
Via the dual bases, the transpose of~\meqref{eq:treelev} can be
identified with the inverse basis map
\begin{equation}
J: V(A) \bij    \bfk\OT^A.
\end{equation}
Therefore, the image of a weight $-1$ monomial in $V(A)$ under $J$ is a planar rooted tree, rather than a linear combination. This is a very different situation compared to~\cite{ZHU2024}. For example,
\[
\begin{array}{@{}c@{\quad}c@{\quad}c@{}}
J(z_1^{a}z_0^{b}z_{-1}^{c}z_{-1}^{d})
=
\TIKZ{%
  \node[dot] (a) at (0,0) {}; \node[below=1pt] at (a) {$a$};
  \node[dot] (b) at (-\TreeXSpread,\TreeYStep) {}; \node[left=\TreeLabelSep] at (b) {$b$};
  \node[dot] (d) at ( \TreeXSpread,\TreeYStep) {}; \node[right=\TreeLabelSep] at (d) {$d$};
  \node[dot] (c) at (-\TreeXSpread,2*\TreeYStep) {}; \node[left=\TreeLabelSep] at (c) {$c$};
  \draw (a)--(b) (a)--(d) (b)--(c);
},
&
J(z_1^{a}z_{-1}^{b}z_0^{c}z_{-1}^{d})
=
\TIKZ{%
  \node[dot] (a) at (0,0) {}; \node[below=1pt] at (a) {$a$};
  \node[dot] (b) at (-\TreeXSpread,\TreeYStep) {}; \node[left=\TreeLabelSep] at (b) {$b$};
  \node[dot] (c) at ( \TreeXSpread,\TreeYStep) {}; \node[right=\TreeLabelSep] at (c) {$c$};
  \node[dot] (d) at ( \TreeXSpread,2*\TreeYStep) {}; \node[left=\TreeLabelSep] at (d) {$d$};
  \draw (a)--(b) (a)--(c) (c)--(d);
},
&
J(z_2^{a}z_{-1}^{b}z_{-1}^{c}z_{-1}^{d})
=
\TIKZ{%
  \node[dot] (a) at (0,0) {}; \node[below=1pt] at (a) {$a$};
  \node[dot] (b) at (-\TreeXSpread,2*\TreeYStep) {}; \node[left=\TreeLabelSep] at (b) {$b$};
  \node[dot] (c) at (0,2*\TreeYStep) {}; \node[above=\TreeLabelSep] at (c) {$c$};
  \node[dot] (d) at ( \TreeXSpread,2*\TreeYStep) {}; \node[right=\TreeLabelSep] at (d) {$d$};
  \draw (a)--(b) (a)--(c) (a)--(d);
}
\end{array}
\,.
\]
\mlabel{eq:Jinverse}
\end{remark}

By the universal property of the free associative algebra $\bfk\OF^A$, the map $\Phi_{\mathrm{fer}}^{\mathcal T}$ in \eqref{eq:treelev}
extends uniquely to an algebra homomorphism
\begin{equation}
\Phi_{\mathrm{fer}}^{\mathcal F}:\bfk\OF^A\longrightarrow \bnca_{\leq 0},
\quad
t_1\cdots t_m \mapsto \Phi_{\mathrm{fer}}^{\mathcal T}(t_1)\cdots \Phi_{\mathrm{fer}}^{\mathcal T}(t_m),
\mlabel{eq:ferF}
\end{equation}
whose restriction to ${\rm Lie}(\bfk \OT^A)$ is a Lie algebra isomorphism
\begin{equation}
\Phi_{\mathrm{fer}}:=\Phi_{\mathrm{fer}}^{\mathcal F}\big|_{{\rm Lie}(\OT^A)}: {\rm PSL}(A) = {\rm Lie}(\bfk \OT^A)\bij  \mathfrak g(A) = {\rm Lie}(V(A)),
\mlabel{Liehom}
\end{equation}
We call the $\Phi_{\mathrm{fer}}$ the {\bf planar fertility map}.

\begin{prop}
For $s,t\in \OT^A$, one has
\begin{equation*}
\Phi_{\mathrm{fer}}^{\mathcal T}(s\curvearrowright_l t)
=
\Phi_{\mathrm{fer}}^{\mathcal T}(s)\btri \Phi_{\mathrm{fer}}^{\mathcal T}(t).
\end{equation*}
\mlabel{treelevel}
\end{prop}

\begin{proof}
Let $V(t)=\{v_1<\cdots <v_q\}$ be the vertex set of $t$, ordered  by  the root-subtree order in Section~\ref{sec:order}. Write
\begin{equation}
\Phi_{\mathrm{fer}}^{\mathcal T}(s)=:\bx,
\quad
\Phi_{\mathrm{fer}}^{\mathcal T}(t)=:\by=y_{1}\cdots y_{q}.
\mlabel{eq:phiphi}
\end{equation}
Since $s\in\OT^A$, one has
\begin{equation}
\wt(\bx)=-1,\quad\partial^{-\wt(\bx)}=\partial.
\mlabel{eq:parfuyi}
\end{equation}
By~\eqref{eq:lgraft},
\[
s\curvearrowright_l t
=
\sum_{j=1}^q s\curvearrowright_{v_j} t.
\]
For each $j$, the grafting at $v_j$ increases the fertility of $v_j$ by one:
$$f(v_j)\longmapsto f(v_j)+1,$$
On the monomial side, this means exactly
\[
y_{j}\longmapsto \partial(y_{j}).
\]
 Moreover, since the new subtree $s$ is
inserted as the leftmost child of $v_j$, its fertility word $\bx$ appears
immediately after $\partial(y_j)$ and before the words corresponding to the
original children of $v_j$.  Therefore
\[
\Phi_{\mathrm{fer}}^{\mathcal T}(s\curvearrowright_{v_j} t)
=
y_1\cdots y_{j-1}\,\partial(y_j)\,\bx\,y_{j+1}\cdots y_q .
\]
Hence
\begin{align*}
\Phi_{\mathrm{fer}}^{\mathcal T}(s\curvearrowright_l t)
&\overset{\eqref{eq:lgraft}}{=}
\sum_{j=1}^q
\Phi_{\mathrm{fer}}^{\mathcal T}(s\curvearrowright_{v_j} t)\\
&\overset{(\ref{eq:treelev})}{=}
\sum_{j=1}^q
y_{1}\cdots y_{j-1}
\,\partial(y_{j})\,
\bx\,
y_{j+1}\cdots y_{q}\\
&\overset{\eqref{eq:parfuyi}}{=}
\sum_{j=1}^q
y_{1}\cdots y_{j-1}
\,\partial^{-\wt(\bx)}(y_{j})\,
\bx\,
y_{j+1}\cdots y_{q}\\
&\overset{(\ref{eq:winser})}{=}
\bx\btri \by\\
&\overset{\eqref{eq:phiphi}}{=}
\Phi_{\mathrm{fer}}^{\mathcal T}(s)\btri \Phi_{\mathrm{fer}}^{\mathcal T}(t).
\end{align*}
\end{proof}

\begin{prop}\mlabel{faen}
The planar fertility map $\Phi_{\mathrm{fer}}$ in~\eqref{Liehom} is a post-Lie algebra isomorphism.
\end{prop}

\begin{proof}
Since
$
\bigl({\rm PSL}(A),\curvearrowright_l,\lpp{-}{-}\bigr)
$
is the free post-Lie algebra on the set $A$, and
$
\bigl(\mathfrak g(A),\btri,\lpp{-}{-}\bigr)
$
is a post-Lie algebra by Theorem~\mref{postPN}, there exists a unique post-Lie algebra homomorphism
\begin{equation}
\lbar{\Phi}:{\rm PSL}(A)\longrightarrow \mathfrak g(A)
\mlabel{eq:tilphi}
\end{equation}
such that
\begin{equation}
\lbar{\Phi}(\bullet_a)=z_{-1}^a=:\Phi_{\mathrm{fer}}^{\mathcal T}(\bullet_a),
\quad \forall\, a\in A.
\mlabel{eq:vertex}
\end{equation}
We are left to prove that $\lbar{\Phi}= \Phi_{\mathrm{fer}}$. Notice that the free post-Lie algebra $\mathrm{PSL}(A)$ is generated by the
one-vertex trees, so every element of $\mathrm{PSL}(A)$ is a linear
combination of iterated expressions built from them by the Lie bracket and
the post-Lie product $\curvearrowright_l$.
So it follows from
\[
\lbar{\Phi}(s\curvearrowright_l t)
=
\lbar{\Phi}(s)\btri \lbar{\Phi}(t),
\quad \forall\, s,t\in\OT^A,
\]
Proposition~\mref{treelevel} and \eqref{eq:vertex} that
\begin{equation*}
\lbar{\Phi}(t)=\Phi_{\mathrm{fer}}^{\mathcal T}(t),
\quad \forall\, t\in\OT^A.
\mlabel{eq:plantree}
\end{equation*}
Finally, in terms of~\meqref{Liehom} and the fact that $\lbar{\Phi}$
in \eqref{eq:tilphi} is also a Lie algebra homomorphism, we have $\lbar{\Phi}=\Phi_{\mathrm{fer}}:{\rm PSL}(A) \longrightarrow \mathfrak g(A)$ as Lie algebra homomorphisms.
\end{proof}

Once the planar tree fertility map in~(\ref{eq:treelev}) identifies monomials with planar rooted trees,
structural identities on the MKW side can be transported to the PLOT side.
The mock-cocycle condition is the first such transported identity.

\subsection{Mock-cocycle condition of $\Delta_{\rm PLOT}$}
The multiplicative extension
\[
\Phi_{\mathrm{fer}}: U\big({\rm PSL}(A)\big)\longrightarrow U\big(\mathfrak g(A)\big)
\]
of the planar fertility map $\Phi_{\mathrm{fer}}$ in~\meqref{Liehom} to universal enveloping algebra is  a Hopf morphism
\[\Phi_{\mathrm{fer}}: \Big(U\big(\rm PSL(A)\big),\star,\Delta_{\shuffle},\etree, \varepsilon\Big)\longrightarrow \Big(U\big(\mathfrak g(A)\big),\star,\Delta_{\shuffle},\etree, \varepsilon\Big),\]
where both ordered Grossman-Larson products are denoted by $\star$.
Taking linear graded duality, $\Phi_{\mathrm{fer}}$ induces a Hopf algebra isomorphism
\begin{equation}
J:\mathcal H_{\mathrm{PLOT}}^{A} \bij  \mathcal H_{\mathrm{MKW}}^{A}.
\mlabel{J:hopfm}
\end{equation}

\noindent Set
\begin{equation}
{\mathcal W}(A):=
\{\bw_1\cdots \bw_r\mid r\ge 0,\ 
\bw_i\in V(A)\text{ for }1\le i\le r\}.
\mlabel{eq:zzzz}
\end{equation}
 For any $a\in A$, we define  the planar multi-index version of the grafting operator
\[
L_a^+:\mathcal W(A)\longrightarrow \mathfrak g(A),
\quad
\bw_{1}\cdots \bw_{r}\longmapsto z_{r-1}^{a}\bw_{1}\cdots \bw_{r}.
\]
Then
\begin{equation*}
\Phi_{\mathrm{fer}}^{\mathcal T}\circ B_a^{+}
=
\la\circ \Phi_{\mathrm{fer}}^{\mathcal F}.
\end{equation*}
Its transpose
$
\lam:\mathcal H_{\rm PLOT}^{A}\longrightarrow \mathcal H_{\rm PLOT}^{A}
$
satisfies that
\begin{equation}
J\circ \lam
=
B_a^{-}\circ J.
\mlabel{eq:JBa}
\end{equation}
For every  monomial $\bzia \in V(A)$, using the Hopf algebra morphism property for $J$:
\begin{align}
(J\otimes J) \circ \Delta_{\rm PLOT}(\bzia)
&=\Delta_{\mathrm{MKW}}\circ\bigl(J(\bzia)\bigr) \notag\\
&\overset{(\ref{eq:cocyle})}{=}J(\bzia)\otimes \etree
+\sum_{a\in A}(\id\otimes B_a^{+})\,
\Delta_{\mathrm{MKW}}\,\bigl(B_a^{-}\circ(J(\bzia))\bigr)\notag\\
&\overset{\eqref{eq:JBa}}{=}J(\bzia)\otimes \etree
+\sum_{a\in A}(\id\otimes B_a^{+})\,
\Delta_{\mathrm{MKW}}\circ\bigl(J(\lam(\bzia))\bigr)\nonumber
\\
&=J(\bzia)\otimes \etree
+\sum_{a\in A}(\id\otimes B_a^{+})\circ
(J\otimes J) \circ \Delta_{\rm PLOT}\bigl(\lam(\bzia)\bigr),\nonumber
\end{align}
which is called the {\bf mock-cocycle condition} of the coproduct $\Delta_{\rm PLOT}$.
\section{The combinatorial explanation of the coproduct}\mlabel{sec:coproduct}
In this section, we first reconstruct the parent-map structure for  monomials  $\bzia\in V(A)$  on the monomial side. We then define the  left-admissible cuts for $\bzia$ and obtain a combinatorial description of the coproduct $\Delta_{\rm PLOT}^{A}$.
\nc{\cc}{\mathcal{c}} For $\bzia := x_{1}\cdots x_{p}$,
denote the number of children of $x_j$ by
\[
{\rm ch}(x_j):=\wt(x_j)+1\in\ZZ_{\ge 0}, \,\text{ where }\, j=1,\ldots,p.
\]

\begin{exam}
Consider the planar $A$-decorated rooted tree
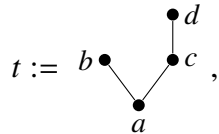
\[
t:=\vcenter{\hbox{\tikz{
  \node[dot] (a) at (0,0) {}; \node[below=1pt] at (a) {$a$};
  \node[dot] (b) at (-\TreeXSpread,\TreeYStep) {}; \node[left=\TreeLabelSep] at (b) {$b$};
  \node[dot] (c) at ( \TreeXSpread,\TreeYStep) {}; \node[right=\TreeLabelSep] at (c) {$c$};
  \node[dot] (d) at ( \TreeXSpread,2*\TreeYStep) {}; \node[right=\TreeLabelSep] at (d) {$d$};
  \draw (a)--(b) (a)--(c) (c)--(d);
}}}
,\]
by \eqref{eq:treelev}, we have
$$
\Phi_{\mathrm{fer}}^{\mathcal T}(t)=
z_{1}^{a}z_{-1}^{b}z_{0}^{c}z_{-1}^{d}=:x_1x_2x_3x_4.
$$
Roughly speaking, $x_1$ has two children $x_2$ and $x_3$, $x_3$ has one child $x_4$, while both $x_2$ and $x_4$ have no children.  Since ${\rm ch}(x_1)=2$, this part  splits into two consecutive  blocks of total weight $-1$, namely $x_2$ and $x_3x_4.$
Therefore the two children of $x_1$ are the first letters of these two blocks:
$x_2$, \,$x_3$. Similarly, since ${\rm ch}(x_3)=1$, the word after $x_3$ inside the block generated by $x_3$ splits into one  block $(x_4)$. So the unique child of $x_3$ is $x_4$.
\mlabel{extree}
\end{exam}

\noindent Let
\[
\bzia=x_1\cdots x_p\in V(A).
\]
For each $1\le j\le p$, define
\begin{equation}
b_{\bzia}(x_j):=\min \Bigl\{ k\ge j \;\Big|\; \sum_{m=j}^{k} {\rm wt}(x_{m})=-1 \Bigr\}.
\mlabel{eq:bj}
\end{equation}
Then set
\begin{equation}
B_{\bzia}(x_j):=x_j x_{j+1}\cdots x_{b_{\bzia}(x_j)},
\quad
B^\circ_{\bzia}(x_j):=x_{j+1}\cdots x_{b_{\bzia}(x_j)},
\mlabel{eq:weightb}
\end{equation}
which are called  the {\bf descendant-generated block} and the {\bf descendant block} of $x_j$, respectively.
Notice that
\begin{equation}
B_{\bzia}(x_1) = \bzia.
\mlabel{eq:bz1}
\end{equation}

\begin{remark}
\begin{enumerate}
\item The set in the right-hand side of~\meqref{eq:bj} is not necessarily a singleton set. For example,  consider $$\bzia=
z_{2}^{a}z_{0}^{b}z_{0}^{c}z_{-1}^{d}z_{-1}^{e}z_{1}^{f}z_{-1}^{g}z_{-1}^{h}
=:x_1x_2x_3x_4x_5x_6x_7x_8.$$ For $j=4$, both $k=4$ and $k=6$ satisfy the condition
$\sum_{m=4}^{k} {\rm wt}(x_m)=-1$.

\item At the tree level, the minimality in~\meqref{eq:bj} ensures that the vertices corresponding to the descendant block $B_{\bzia}(x_j)^{\circ}$ are precisely all the descendants of the vertex corresponding to $x_j$.
\end{enumerate}
\end{remark}

We next define the  child sets and the parent map. Let $1\leq j \leq p$ be arbitrary fixed.
If ${\rm ch}(x_j)=0$, then define
$
{\rm Ch}_{\bzia}(x_j):=\emptyset.
$
If ${\rm ch}(x_j)=r\ge 1$,
inside $B_{\bzia}(x_j)$, define recursively on $1\le s\le r-1$ that
\begin{equation}
c_{j,1}:=j+1,\quad
\ell_{j,s}:= b_{\bzia}(x_{c_{j,s}}),\quad
c_{j,s+1}:=\ell_{j,s}+1.
\mlabel{eq:childstart}
\end{equation}
Roughly speaking, the $c_{j,s}$ represents the position of the $s$-th child of $x_j$, and the $\ell_{j,s}$ represents the position immediately preceding the $(s + 1)$-th child of $x_j$.
Now we define the ordered child set of $x_j$ by
\begin{equation}
{\rm Ch}_{\bzia}(x_j):=
\bigl\{x_{c_{j,1}} < \dots < x_{c_{j,r}}\bigr\}.
\mlabel{eq:childset}
\end{equation}
Then $B_{\bzia}(x_j)$ admits the decomposition
\begin{equation}
B_{\bzia}(x_j) = x_j B_{\bzia}(x_{c_{j,1}})\cdots  B_{\bzia}(x_{c_{j,r}}), \quad B^\circ_{\bzia}(x_j) =  B_{\bzia}(x_{c_{j,1}})\cdots  B_{\bzia}(x_{c_{j,r}}).
\mlabel{eq:blockdecomp}
\end{equation}
Each $B_{\bzia}(x_{c_{j,s}})$ is a consecutive subword of $B_{\bzia}(x_j)$. Further, since
\[
c_{j,s+1}=\ell_{j,s}+1 = b_{\bzia}(x_{c_{j,s}}) +1 >c_{j,s},
\]
we have
\[
c_{j,1}<c_{j,2}<\cdots<c_{j,r},
\]
and so $B_{\bzia}(x_{c_{j,1}}), \ldots,  B_{\bzia}(x_{c_{j,r}})$ are pairwise disjoint consecutive subwords of $B_{\bzia}(x_j)$.\\

The next lemma is the basic reconstruction statement: every non-root element
has one and only one parent determined by the block decomposition.

\begin{lemma}
With the above setting,
\begin{equation}
\{x_2,\dots,x_p\}=\bigsqcup_{j=1}^p {\rm Ch}_{\bzia}(x_j). \mlabel{eq:disj}
\end{equation}
\mlabel{lem:disju}
\end{lemma}

\begin{proof}
For brevity, write
\[
{\bf z}_{\mathbf{i}}^{\mathbf{a}}:=x_1\cdots x_p,\quad
b(j):=b_{{\bf z}_{\mathbf{i}}^{\mathbf{a}}}(x_j),\quad
B_j:=B_{{\bf z}_{\mathbf{i}}^{\mathbf{a}}}(x_j),\quad
B_j^\circ:=B_{{\bf z}_{\mathbf{i}}^{\mathbf{a}}}^\circ(x_j),\quad
{\rm Ch}_j:= {\rm Ch}_{\bzia}(x_j).
\]
Thus $B_j=x_jx_{j+1}\cdots x_{b(j)}$ by (\ref{eq:weightb}). We first establish the following auxiliary fact:
\begin{equation}
j<k\le b(j)\quad\Longrightarrow\quad b(k)\le b(j)\,\text{ and so }\, B_k\subsetneq B_j.
\mlabel{eq:lem33nesting}
\end{equation}
Indeed, by the minimality in (\ref{eq:bj}), for every $n$ with $j\le n<b(j)$ one has
\begin{equation}
\sum_{m=j}^n \mathrm{wt}(x_m)\ge 0.
\mlabel{eq:lem33prefix}
\end{equation}
Otherwise, if for some $n<b(j)$ the sum is less than or equal to $-1$, then since each
$\mathrm{wt}(x_m)\in \mathbb{Z}_{\ge -1}$, the sequence of partial sums would
have to pass through the value $-1$ before reaching that value, contradicting
the definition (\ref{eq:bj}) of $b(j)$.
Now let $j<k\le b(j)$. Using (\ref{eq:lem33prefix}) with $n=k-1$, we obtain
\[
\sum_{m=k}^{b(j)} \mathrm{wt}(x_m)
=
\sum_{m=j}^{b(j)} \mathrm{wt}(x_m)-\sum_{m=j}^{k-1}\mathrm{wt}(x_m)
=
-1-\sum_{m=j}^{k-1}\mathrm{wt}(x_m)
\le -1.
\]
Again, because each increment is at least $-1$, the partial sums
$\sum_{m=k}^n \mathrm{wt}(x_m)$ must hit the value $-1$ for some $n\le b(j)$.
By (\ref{eq:bj}), this implies $b(k)\le b(j)$ and so $B_k\subsetneq B_j$, proving
(\ref{eq:lem33nesting}).

\noindent\textbf{Step 1.} We show that
\begin{equation}
\{x_2,\ldots,x_p\}=\bigcup_{j=1}^p {\rm Ch}_j.
\mlabel{eq:lem33cover}
\end{equation}

\noindent The inclusion
\[
\bigcup_{j=1}^p {\rm Ch}_j\subseteq \{x_2,\ldots,x_p\}
\]
is immediate from (\ref{eq:childset}), since every child index is of the form
$p\geq c_{j,s}\ge j+1\ge 2$.
Conversely, fix $k\in\{2,\ldots,p\}$. Since $B_1=x_1\cdots x_p$ by
(\ref{eq:bz1}), we have $x_k\in B_1^\circ$. Hence the family
\[
\mathcal{F}_k:=\{\, B_j \mid 1\le j<k,\ x_k\in B_j^\circ \,\}
\]
is nonempty. Choose $B_{j_0}\in \mathcal{F}_k$ minimal with respect to
inclusion. Let $r:={\rm ch}(x_{j_0})\geq1$. By (\ref{eq:blockdecomp}),
\[
x_k\in B^\circ_{j_0}= B_{c_{j_0,1}}\cdots B_{c_{j_0,r}},
\]
where the blocks $B_{c_{j_0,1}},\ldots,B_{c_{j_0,r}}$ are pairwise disjoint
consecutive subwords of $B_{j_0}$.
Therefore $x_k$ belongs to exactly one of these blocks, say
$x_k\in B_{c_{j_0,s}}$ for some $s$. We claim that $k=c_{j_0,s}$. If not,
then $k>c_{j_0,s}$, so $x_k\in B_{c_{j_0,s}}^\circ$. Hence
$B_{c_{j_0,s}}\in \mathcal{F}_k$. Moreover, (\ref{eq:lem33nesting}) gives
\[
B_{c_{j_0,s}}\subsetneq B_{j_0},
\]
contradicting the minimality of $B_{j_0}$. Thus $k=c_{j_0,s}$, and therefore
$x_k\in {\rm Ch}_{j_0}$ by (\ref{eq:childset}).

\noindent\textbf{Step 2.} We prove that the union in (\ref{eq:lem33cover}) is disjoint. It is enough to show that
\begin{equation}
{\rm Ch}_j\cap {\rm Ch}_{j'}=\varnothing
\quad\text{for all } j\ne j'.
\mlabel{eq:lem33pairwise}
\end{equation}
Assume, for contradiction, that there exists
$x_k\in {\rm Ch}_j\cap {\rm Ch}_{j'}$ with $j\ne j'$. Without loss of generality, suppose
$j<j'$. Since $x_k\in {\rm Ch}_j$, there exists $s$ such that $k=c_{j,s}$. Hence,
by (\ref{eq:blockdecomp}),
\[
B_j
=
x_jB_{c_{j,1}}\cdots B_{c_{j,s-1}}B_kB_{c_{j,s+1}}\cdots B_{c_{j,r}}.
\]
Because the child blocks are consecutive and pairwise disjoint, every letter
strictly between $x_j$ and the first letter $x_k$ of the block $B_k$ must lie
in one of the earlier blocks
$B_{c_{j,1}},\ldots,B_{c_{j,s-1}}$. Since $j<j'<k$, there exists some $1\leq u<s$
such that
\begin{equation}
x_{j'}\in B_{c_{j,u}}.
\mlabel{eq:zjpj}
\end{equation}

If $j'=c_{j,u}$, then $B_{j'}=B_{c_{j,u}}$, which lies entirely to the left of
$B_k$ because $u<s$. If $j'>c_{j,u}$, then $b(c_{j,u})\geq j' > c_{j,u}$ by \meqref{eq:zjpj}, and so (\ref{eq:lem33nesting}) gives that
\[
B_{j'}\subsetneq B_{c_{j,u}}.
\]
Thus again $B_{j'}$ lies entirely to the left of $B_k$. In both cases, the last
letter of $B_{j'}$ occurs before $x_k$, and therefore $x_k\notin B_{j'}$.
But this is impossible, because $x_k\in {\rm Ch}_{j'}$ means that $x_k$ is the first
letter of one of the descendant-generated blocks appearing in the decomposition
(\ref{eq:blockdecomp}) of $B_{j'}$, hence certainly $x_k\in B_{j'}$. This
contradiction proves (\ref{eq:lem33pairwise}). Combining Step $1$ and Step $2$, we obtain (\ref{eq:disj}).
\end{proof}

\noindent Thanks to~\meqref{eq:childset} and~\meqref{eq:disj}, we now can define the {\bf parent map}
\begin{equation*}
\Par_{\bzia}:\{x_2<\cdots<x_p\}\longrightarrow\{x_1,\ldots,x_{p-1}\}, \quad x_{c_{j,s}} \mapsto x_j.
\end{equation*}

\begin{exam}
Consider
\[
\bzia=
z_{1}^{a}\, z_{-1}^{b}\, z_{0}^{c}\, z_{-1}^{d}
=:x_1x_2x_3x_4.
\]
in Example~\ref{extree}, which  can obtained by $\Phi_{\rm fer}^{\mathcal T}$ acting on  the planar rooted tree:
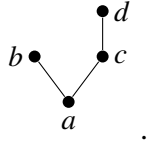
\[
\tikz{
  \node[dot] (a) at (0,0) {}; \node[below=1pt] at (a) {$a$};
  \node[dot] (b) at (-\TreeXSpread,\TreeYStep) {}; \node[left=\TreeLabelSep] at (b) {$b$};
  \node[dot] (c) at ( \TreeXSpread,\TreeYStep) {}; \node[right=\TreeLabelSep] at (c) {$c$};
  \node[dot] (d) at ( \TreeXSpread,2*\TreeYStep) {}; \node[right=\TreeLabelSep] at (d) {$d$};
  \draw (a)--(b) (a)--(c) (c)--(d);
}.
\]
We have
\[
({\rm ch}(x_1),{\rm ch}(x_2),{\rm ch}(x_3),{\rm ch}(x_4))=(2,0,1,0).
\]
By~\eqref{eq:bj},
\[
b_{\bzia}(x_1)=4,\quad b_{\bzia}(x_2)=2,\quad b_{\bzia}(x_3)=4,\quad b_{\bzia}(x_4)=4.
\]
Hence
\[
B_{\bzia}(x_1)=x_1x_2x_3x_4,\quad
B_{\bzia}(x_2)=x_2,\quad
B_{\bzia}(x_3)=x_3x_4,\quad
B_{\bzia}(x_4)=x_4.
\]

\noindent We now determine the child sets.
Since ${\rm ch}(x_1)=2$, by~\eqref{eq:childstart},
\[
c_{1,1}=2, \quad\ell_{1,1}=b_{\bzia}(x_2)=2,\quad
c_{1,2}=3, \quad\ell_{1,2}=b_{\bzia}(x_3)=4.
\]
Thus
\[
B_{\bzia}(x_1)=x_1\,B_{\bzia}(x_2)\,B_{\bzia}(x_3),\quad
{\rm Ch}_{\bzia}(x_1)=\{ x_2 < x_3\}.
\]
Since ${\rm ch}(x_3)=1$,
$c_{3,1}=4$ and $\ell_{3,1}=b_{\bzia}(x_4)=4$,
we have
\[
B_{\bzia}(x_3)=x_3\,B_{\bzia}(x_4),\quad
B_{\bzia}(x_4)=x_4,\quad
{\rm Ch}_{\bzia}(x_3)=\{ x_4\}.
\]
Therefore
\[
\{x_2< x_3< x_4\}
=
{\rm Ch}_{\bzia}(x_1)\sqcup{\rm Ch}_{\bzia}(x_3),
\]
and
$$
\Par_{\bzia}(x_2)=x_1,\quad
\Par_{\bzia}(x_3)=x_1,\quad
\Par_{\bzia}(x_4)=x_3.
$$
\mlabel{ex:pardetailed}
\end{exam}

\begin{defn}
Let $\bzia=x_1\cdots x_p$ be equipped with its  parent map ${\rm Par}_{\bzia}$ and its child sets ${\rm Ch}_{\bzia}(x_j)$.
A  subset  $\xhc \subseteq \{x_2,\dots,x_p\}$ is called a \textbf{(non-total) left-admissible cut} of $\bzia$
if it satisfies the following two conditions:
\begin{enumerate}
\item[(1)] (\textbf{Leftmost-children rule})
For every $j\in\{1,\dots,p\}$, the cut elements among the children of $x_j$ must form an initial segment of the ordered child set ${\rm Ch}_{\bzia}(x_j)$.
More precisely, if
\[
{\rm Ch}_{\bzia}(x_j) =\bigl(x_{c_{j,1}},\dots,x_{c_{j,r}}\bigr),
\]
then there exists an integer $m_j\in\{0,1,\dots,r\}$ such that
\[
 \mathfrak{c} \cap {\rm Ch}_{\bzia}(x_j)=\{x_{c_{j,1}},\dots,x_{c_{j,m_j}}\}.
\]

\item[(2)] (\textbf{No two elements in the same descendant-generated block})
For any two distinct elements $x_{j},x_{k} \in \xhc $, one has
\[
x_{j} \notin B^\circ_{\bzia}(x_{k}),
\quad
x_{k}\notin B^\circ_{\bzia}(x_{j}).
\]
\end{enumerate}
\mlabel{def:LAdmPSTS}
\end{defn}

Note that the empty cut  is left-admissible. By convention, we also add the total cut $\{x_1,\dots,x_p\}$ as a left-admissible cut.  
We denote by $\mathrm{LAdm}(\bzia)$ the set of left-admissible cuts 
$\xhc  $. In the following construction of $P^{\xhc}(\bzia)$ and $R^{\xhc}(\bzia)$, we
assume that $\xhc$ is neither the empty cut nor the total cut. Let the left-admissible cut
$\xhc=\{x_{j_1}<\cdots<x_{j_r}\}\in \mathrm{LAdm}(\bzia),$
be neither the empty cut nor the total cut. 
For each cut vertex
$x_{j_\ell}\in\xhc$, define the {\bf pruned block} at $x_{j_\ell}$ by
\[
B_{\bzia}^{\xhc}(x_{j_\ell})
:=
x_{j_\ell}\,x_{j_\ell+1}\cdots x_{b_{\bzia}(x_{j_\ell})}.
\]  
Since $\xhc$ is left-admissible, the
blocks
$B_{\bzia}^{\xhc}(x_{j_1}),\ldots,B_{\bzia}^{\xhc}(x_{j_r})$
are pairwise disjoint.
We set
\[
{\mathcal P}(\xhc):=
\bigl\{B_{\bzia}^{\xhc}(x_{j_\ell})\mid 1\le \ell\le r\bigr\}.
\]
The pruned part $P^\xhc(\bzia)$ is constructed from ${\mathcal P}(\xhc)$ as follows.
These blocks are first ordered by increasing indices
$j_1<\cdots<j_r.$
Among them, the blocks whose corresponding cut vertices have the same parent are concatenated in this order, while the factors corresponding to different parents are combined by $\shuffle$.
The resulting element is denoted by $P^\xhc(\bzia).$\\

Let ${\mathcal R}(\xhc)$ be the ordered subword obtained from
$x_1\cdots x_p$ by deleting all word elements contained in the pruned blocks
\[
B_{\bzia}^{\xhc}(x_{j_\ell}),
\quad 1\le \ell\le r.
\]
Write
\[
{\mathcal R}(\xhc)=\{x_{r_1}<\cdots<x_{r_s}\},
\]
where the order is inherited from the original word $x_1\cdots x_p$.
For each ${x_j}\in {\mathcal R}(\xhc )$, we define the {\bf modified weight map}
\[
 {\rm wt}_{\xhc}: {\mathcal R}(\xhc ) \longrightarrow  \mathbb Z_{\ge -1}, \quad  x_{j} \mapsto {\rm wt(x_{j})}-|\xhc \cap {\rm Ch}_{\bzia}(x_j)|,
\]
and the product
\[
R^{\xhc}(\bzia)
:=\prod_{{x_j}\in {\mathcal R}(\xhc )}
z_{{\rm wt}_{\xhc}(x_{j})}^{d(x_{j})}.
\]

The two exceptional cuts are understood by convention.  If $\xhc$ is the
empty cut, then no block is pruned, and we set
\[
P^\xhc(\bzia):=\etree,
\quad
R^\xhc(\bzia):=\bzia.
\]
If $\xhc=\{x_1,\ldots,x_p\}$ is the total cut, then the whole word is
pruned, and we set
\[
P^\xhc(\bzia):=\bzia,
\quad
R^\xhc(\bzia):=\etree.
\]

\begin{exam}\label{ex:LAdm-4letters}
Consider the monomial in Example~\ref{ex:pardetailed}
\[
\bzia=
z_{1}^{a}\, z_{-1}^{b}\, z_{0}^{c}\, z_{-1}^{d}
=:x_1x_2x_3x_4.
\]
The parent map is
\[
\Par_{\bzia}(x_1)=0,\quad
\Par_{\bzia}(x_2)=x_1,\quad
\Par_{\bzia}(x_3)=x_1,\quad
\Par_{\bzia}(x_4)=x_3.
\]
Then
\begin{center}
\renewcommand{\arraystretch}{1.18}
\setlength{\tabcolsep}{8pt}
\resizebox{0.98\linewidth}{!}{
\begin{tabular}{|>{\centering\arraybackslash}m{0.15\linewidth}
                |>{\centering\arraybackslash}m{0.15\linewidth}
                |>{\centering\arraybackslash}m{0.24\linewidth}
                |>{\centering\arraybackslash}m{0.22\linewidth}
                |>{\centering\arraybackslash}m{0.13\linewidth}|}
\hline
\textbf{Chosen set $\xhc$} &
\textbf{$\mathcal P(\xhc)$} &
\textbf{$P^\xhc(\bzia)$} &
\textbf{$R^\xhc(\bzia)$} &
\textbf{Left-admissible?}\\
\hline

$\emptyset$ &
$\emptyset$ &
$\etree$ &
$z_{1}^{a}z_{-1}^{b}z_{0}^{c}z_{-1}^{d}$ &
Yes\\
\hline

$\{x_2\}$ &
$\{x_2\}$ &
$z_{-1}^{b}$ &
$z_{0}^{a}z_{0}^{c}z_{-1}^{d}$ &
Yes\\
\hline

$\{x_3\}$ &
$\{x_3,x_4\}$ &
\NAcell &
\NAcell &
No\\
\hline

$\{x_4\}$ &
$\{x_4\}$ &
$z_{-1}^{d}$ &
$z_{1}^{a}z_{-1}^{b}z_{-1}^{c}$ &
Yes\\
\hline

$\{x_2,x_3\}$ &
$\{x_2,x_3,x_4\}$ &
$z_{-1}^{b}(z_{0}^{c}z_{-1}^{d})$ &
$z_{-1}^{a}$ &
Yes\\
\hline

$\{x_2,x_4\}$ &
$\{x_2,x_4\}$ &
$z_{-1}^{b}\shuffle z_{-1}^{d}$ &
$z_{0}^{a}z_{-1}^{c}$ &
Yes\\
\hline

$\{x_3,x_4\}$ &
$\{x_3,x_4\}$ &
\NAcell &
\NAcell &
No\\
\hline

$\{x_2,x_3,x_4\}$ &
$\{x_2,x_3,x_4\}$ &
\NAcell &
\NAcell &
No\\
\hline

$\{x_1,x_2,x_3,x_4\}$ &
$\{x_1,x_2,x_3,x_4\}$ &
$z_{1}^{a}z_{-1}^{b}z_{0}^{c}z_{-1}^{d}$ &
$\etree$ &
Yes\\
\hline
\end{tabular}}
\end{center}
\end{exam}

\begin{lemma}
Let $\bzia\in V(A)$ and
$
t:=J(\bzia)\in \OT^A.
$
Then the map
\[
\mathfrak{C}:\mathrm{LAdm}(t) \longrightarrow \mathrm{LAdm}(\bzia),\quad c \mapsto  \xhc :=\{x_{j} \mid v_j \text{ is an outcoming vertex of a cut edge in $c$} \}
\]
is a bijection. Moreover, for every $c\in \mathrm{LAdm}(t)$,  one has
\[
\Phi_{\mathrm{fer}}^{\mathcal F}\bigl(P^c(t)\bigr)=P^{\xhc}
(\bzia),
\quad
\Phi_{\mathrm{fer}}^{\mathcal T}\bigl(R^c(t)\bigr)=R^{\xhc}
(\bzia).
\]
\mlabel{lem:cut}
\end{lemma}

\begin{proof}
It follows from Proposition~\ref{pro:Phibi}, Remark~\ref{eq:Jinverse}, and the definition of  Left admissible cut on planar rooted trees and Definition~\mref{def:LAdmPSTS}. Thus
\[
c\in \mathrm{LAdm}(t)
\Longleftrightarrow
\xhc\in \mathrm{LAdm}(\bzia)
\]
and $\dhc$ is a bijection. Let $\xhc=\dhc(c)$. By  the definitions of $P^c(t)$ and $P^\xhc(\bzia)$,  the definitions of $R^c(t)$ and $R^\xhc(\bzia)$, we have
\[
\Phi_{\mathrm{fer}}^{\mathcal F}\bigl(P^c(t)\bigr)=P^{\xhc}(\bzia),
\quad
\Phi_{\mathrm{fer}}^{\mathcal T}\bigl(R^c(t)\bigr)=R^{\xhc}(\bzia).
\]
\end{proof}

\begin{theorem}
For every  $\bzia\in V(A)$, one has
\begin{equation*}
\Delta_{\rm PLOT}(\bzia)
=\sum_{\xhc \in \mathrm{LAdm}(\bzia)} P^\xhc(\bzia)\otimes R^\xhc(\bzia).
\end{equation*}
\mlabel{thm:plotcoprodladm}
\end{theorem}

\begin{proof}
Let
$
t:=J(\bzia)\in \OT^A.
$
Then
$
\bzia=\Phi_{\mathrm{fer}}^{\mathcal T}(t).
$
Applying $\Phi_{\mathrm{fer}}^{\mathcal F}\otimes \Phi_{\mathrm{fer}}^{\mathcal T}$ to both sides of ~\eqref{eq:mkwe}, and by the transported definition of the coproduct on the monomial side, we obtain
\begin{align*}
\Delta_{\rm PLOT}\left(\Phi_{\mathrm{fer}}^{\mathcal T}(t)\right)
=&\,
(\Phi_{\mathrm{fer}}^{\mathcal F}\otimes \Phi_{\mathrm{fer}}^{\mathcal T})\Delta_{\mathrm{MKW}}(t)
\\
=&\,
\sum_{c\in \mathrm{LAdm}(t)}
\Phi_{\mathrm{fer}}^{\mathcal F}\bigl(P^c(t)\bigr)\otimes
\Phi_{\mathrm{fer}}^{\mathcal T}\bigl(R^c(t)\bigr)
\quad  (\text{by}~\Phi_{\mathrm{fer}}^{\mathcal T}(t)=\bzia)
\\
=&\,
\sum_{c\in \mathrm{LAdm}(t)}
P^{\dhc(c)}(\bzia)\otimes R^{\dhc(c)}(\bzia)
\quad (\text{by Lemma~\ref{lem:cut}})
\\
=&\,
\sum_{\mathfrak c\in \mathrm{LAdm}(\bzia)}
P^{\mathfrak c}(\bzia)\otimes R^{\mathfrak c}(\bzia)
\quad (\text{by Lemma~\ref{lem:cut}}).
\mlabel{eq:plotcutproof}
\end{align*}
This completes the proof.
\end{proof}
\section{Extraction-contraction}\label{sec:ec}
In this section, we formulate the extraction--contraction map on the side of
planar noncommutative monomials and show that it is transported from the
corresponding ordered-forest construction on the planar rooted-tree side.

\begin{definition}
Let $t\in\OT^A$.  A {\bf covering subforest} $s\subseteq t$ is a partition
of $V(t)$ into connected blocks.  Each block inherits from $t$ the structure
of an $A$-decorated planar rooted tree.  Ordering the connected components
by the root-subtree order of their roots, we regard $s$ as an ordered
$A$-decorated planar forest.\\

For each connected component $u$ of $s$, let $\operatorname{rt}(u)$ denote
its root.  The {\bf contracted tree} $t/s$ is obtained by shrinking each
connected component $u$ of $s$ to a single vertex $\overline u$, preserving
the planar order induced from $t$.  The contracted vertex $\overline u$ is
decorated by $d(\operatorname{rt}(u))$.\\

Following the extraction--contraction coproduct in \mcite{Calaque2011} and Rahm's ordered-forest contraction coaction~\mcite{Rahm2022}, we
use the following extraction--contraction map on the MKW side:
\begin{equation}
\Gamma_{\mathrm{MKW}}:\bfk\OT^A\longrightarrow \bfk\OF^A\otimes\bfk\OT^A,
\quad
t\mapsto \sum_{s\in\mathrm{Cov}(t)}s\otimes t/s,
\mlabel{eq:GammaMKW-ecA}
\end{equation}
where $\mathrm{Cov}(t)$ denotes the set of covering subforests of $t$.
We extend $\Gamma_{\mathrm{MKW}}$ multiplicatively to ordered planar forests by
\[
\begin{aligned}
\Gamma_{\mathrm{MKW}}:\bfk\OF^A
&\longrightarrow \bfk\OF^A\otimes\bfk\OF^A,\\
\etree
&\longmapsto \etree\otimes\etree,\\
t_1\cdots t_m
&\longmapsto
\Gamma_{\mathrm{MKW}}(t_1)\cdots\Gamma_{\mathrm{MKW}}(t_m),
\quad t_1,\ldots,t_m\in\OT^A.
\end{aligned}
\]
Then
\begin{equation}
\Gamma_{\mathrm{MKW}}(t_1\cdots t_m)
=
\sum_{s_i\in\mathrm{Cov}(t_i),\,1\le i\le m}
s_1\cdots s_m
\otimes
(t_1/s_1)\cdots(t_m/s_m).
\mlabel{eq:GammaMKW-forest-explicit}
\end{equation}
\mlabel{def:ec-treeA}
\end{definition}

The preceding definition describes the extraction--contraction construction on the ordered
planar rooted-forest side; we now translate the same structure to the planar LOT side by
expressing covering subforests of $J(z_{\mathbf i}^{\mathbf a})$ as connected partitions of the
letter occurrences of the corresponding noncommutative monomial.

\begin{definition}
Let
\[
\bzia=x_1\cdots x_p
=
z_{i_1}^{a_1}\cdots z_{i_p}^{a_p}\in V(A),
\quad
t_{\bzia}:=J(\bzia)\in\OT^A.
\]
Write
\[
V(t_{\bzia})=\{v_1<\cdots<v_p\},
\]
where the vertex $v_j$ corresponds to the letter occurrence $x_j$.

\begin{enumerate}
\item A nonempty subset $\beta\subseteq\{x_1,\ldots,x_p\}$ is called a
{\bf connected block} if the induced subgraph of $t_{\bzia}$ on
\[
V(\beta):=\{v_j\mid x_j\in\beta\}
\]
is connected.

\item A {\bf covering partition} of $\bzia$ is a partition of
$\{x_1,\ldots,x_p\}$ into connected blocks.  Ordering the blocks by
increasing minima, we write such a partition as
\[
\pi=(\beta_1,\ldots,\beta_m),
\quad
\min\beta_1<\cdots<\min\beta_m.
\]
We denote by $\mathrm{Cov}(\bzia)$ the set of covering partitions of
$\bzia$.

\item For $x_j\in\beta_\ell$, define
\[
\operatorname{out}_\pi(x_j)
:=
\#\{y\in{\rm Ch}_{\bzia}(x_j)\mid y\notin\beta_\ell\},
\]
where ${\rm Ch}_{\bzia}(x_j)$ is the ordered child set introduced in
\meqref{eq:childset}.

\item If
\[
\beta_\ell=\{x_{j_1}<\cdots<x_{j_q}\},
\]
set
\begin{equation}
\rho(\beta_\ell):=d(x_{j_1})=d(v_{j_1}),
\mlabel{eq:ec-rootdecA}
\end{equation}
and define the {\bf reduced block word}
\begin{equation}
\beta_{\ell,\pi}^{\circ}
:=
z^{d(x_{j_1})}_{{\rm wt}(x_{j_1})-\operatorname{out}_\pi(x_{j_1})}
\cdots
z^{d(x_{j_q})}_{{\rm wt}(x_{j_q})-\operatorname{out}_\pi(x_{j_q})}.
\mlabel{eq:ec-redblockA}
\end{equation}
Then define the ordered monomial forest
\[
M_0(\pi):=\beta_{1,\pi}^{\circ}\cdots\beta_{m,\pi}^{\circ}
\]
and the {\bf contracted word}
\[
\bzia/\pi
:=
z^{\rho(\beta_1)}_{{\rm wt}(\beta_1)}
\cdots
z^{\rho(\beta_m)}_{{\rm wt}(\beta_m)},
\quad
{\rm wt}(\beta_\ell):=
\sum_{x_j\in\beta_\ell}{\rm wt}(x_j).
\]
\end{enumerate}
\end{definition}

Notice that every index occurring in \eqref{eq:ec-redblockA} is at least
$-1$, since
\[
\operatorname{out}_\pi(x_j)\le {\rm ch}(x_j)={\rm wt}(x_j)+1.
\]

\begin{lemma}
For every $\pi\in\mathrm{Cov}(\bzia)$ and every block $\beta\in\pi$, one has
\begin{equation}
{\rm wt}(\beta)
=
\sum_{x_j\in\beta}\operatorname{out}_\pi(x_j)-1.
\mlabel{eq:ec-blockweightA}
\end{equation}
Consequently, for each block $\beta_\ell\in\pi$,
\begin{equation*}
\beta_{\ell,\pi}^{\circ}\in V(A),
\mlabel{eq:ec-blockminusoneA}
\end{equation*}
and also
\begin{equation*}
\bzia/\pi\in V(A).
\mlabel{eq:ec-quotient-in-bnca}
\end{equation*}
\mlabel{lem:ec-blockA}
\end{lemma}

\begin{proof}
Let $\beta\in\pi$ and set $m:=|\beta|$.  Since $\beta$ is connected, the
induced subgraph on $V(\beta)$ is a rooted tree with $m$ vertices and
$m-1$ internal edges.  For $x_j\in\beta$, the number
${\rm wt}(x_j)+1$ is the number of children of the corresponding vertex
$v_j$ in the ambient tree $t_{\bzia}$.  Summing over all vertices in
$\beta$, the internal edges contribute $m-1$, while the edges leaving the
block contribute
\[
\sum_{x_j\in\beta}\operatorname{out}_\pi(x_j).
\]
Thus
\[
\sum_{x_j\in\beta}\bigl({\rm wt}(x_j)+1\bigr)
=
(m-1)+\sum_{x_j\in\beta}\operatorname{out}_\pi(x_j).
\]
Subtracting $m$ gives \eqref{eq:ec-blockweightA}.

Let $\beta=\beta_\ell=\{x_{j_1}<\cdots<x_{j_q}\}$.  By
\eqref{eq:ec-redblockA} and \eqref{eq:ec-blockweightA},
\[
{\rm wt}(\beta_{\ell,\pi}^{\circ})
=
\sum_{u=1}^q
\Bigl({\rm wt}(x_{j_u})-\operatorname{out}_\pi(x_{j_u})\Bigr)
=
-1.
\]
Moreover, $\beta_{\ell,\pi}^{\circ}$ is the planar fertility word of the
connected planar subtree induced by the block $\beta_\ell$.  Hence it
satisfies the proper-prefix condition, and therefore
\[
\beta_{\ell,\pi}^{\circ}\in V(A).
\]

Finally,
\[
{\rm wt}(\bzia/\pi)
=
\sum_{\ell=1}^m{\rm wt}(\beta_\ell)
=
{\rm wt}(\bzia)
=
-1.
\]
The word $\bzia/\pi$ is the planar fertility word of the contracted tree
$t_{\bzia}/s_\pi$, where $s_\pi$ is the covering subforest corresponding to
$\pi$.  Hence $\bzia/\pi$ also satisfies the proper-prefix condition, and so
\[
\bzia/\pi\in V(A).
\]
\end{proof}

\noindent We now define the {\bf extraction--contraction map} on the planar LOT side.  For
$\bzia\in V(A)$, set
\begin{equation}
\Gamma_{\mathrm{PLOT}}(\bzia)
:=
\sum_{\pi\in\mathrm{Cov}(\bzia)}
M_0(\pi)\otimes\bzia/\pi,
\mlabel{eq:GammaPLOT-ecA}
\end{equation}
where the first tensor factor is viewed as an ordered monomial forest. Let
\[
\mathcal W(A):=
\{\bw_1\cdots\bw_m\mid m\ge0,\ \bw_i\in\ V(A)\}
\]
be the set of ordered monomial forests.  We extend $\Gamma_{\mathrm{PLOT}}$ multiplicatively to ordered monomial
forests:
\[
\Gamma_{\mathrm{PLOT}}:\bfk\mathcal W(A)
\longrightarrow
\bfk\mathcal W(A)\otimes\bfk\mathcal W(A)
\]
by setting
\[
\etree\longmapsto \etree\otimes\etree,
\quad
\bw_1\cdots\bw_m
\longmapsto
\Gamma_{\mathrm{PLOT}}(\bw_1)\cdots
\Gamma_{\mathrm{PLOT}}(\bw_m).
\]
Equivalently,
\begin{equation}
\Gamma_{\mathrm{PLOT}}(\bw_1\cdots\bw_m)
=
\sum_{\pi_i\in\mathrm{Cov}(\bw_i),\,1\le i\le m}
M_0(\pi_1)\cdots M_0(\pi_m)
\otimes
(\bw_1/\pi_1)\cdots(\bw_m/\pi_m).
\mlabel{eq:GammaPLOT-forest-explicit}
\end{equation}

Let $J_{\mathcal F}$ be the multiplicative extension of $J$:
\[
J_{\mathcal F}:\bfk\mathcal W(A)\longrightarrow\bfk\OF^A
\]
defined by
\[
\etree\longmapsto \etree,
\quad
\bw_1\cdots\bw_m
\longmapsto
J(\bw_1)\cdots J(\bw_m),
\quad
\bw_i\in V(A).
\]

The definitions above were designed to mirror covering subforests and contractions under the fertility correspondence. The next theorem confirms this
design.
\begin{theorem}
With the above definitions, one has
\begin{equation}
(J_{\mathcal F}\otimes J_{\mathcal F})
\circ \Gamma_{\mathrm{PLOT}}
=
\Gamma_{\mathrm{MKW}}\circ J_{\mathcal F}
\mlabel{eq:ec-intertwineA}
\end{equation}
as maps
\[
\bfk\mathcal W(A)\longrightarrow \bfk\OF^A\otimes\bfk\OF^A.
\]
In other words, \eqref{eq:GammaPLOT-ecA} is the word-side
extraction--contraction formula transported from the $A$-decorated planar
rooted-forest side.
\mlabel{thm:ec-transportA}
\end{theorem}

\begin{proof}
By multiplicativity, it is sufficient to evaluate both sides of \eqref{eq:ec-intertwineA} on $V(A)$. 
Let $\bzia=x_1\cdots x_p\in V(A)$ and set $t:=J(\bzia)$.  Write
\[
V(t)=\{v_1<\cdots<v_p\},
\]
where $v_j$ corresponds to the letter occurrence $x_j$. A covering subforest $s\in\mathrm{Cov}(t)$ determines a covering partition
$\pi_s\in\mathrm{Cov}(\bzia)$ by putting $x_j$ and $x_k$ in the same block
precisely when $v_j$ and $v_k$ belong to the same connected component of
$s$.  Conversely, every covering partition $\pi\in\mathrm{Cov}(\bzia)$
determines a covering subforest $s_\pi\subseteq t$ by taking the induced
connected subtrees associated with its blocks.  This gives a bijection
\begin{equation}\label{eq:ec-bijectionA}
\mathrm{Cov}(t)\longleftrightarrow \mathrm{Cov}(\bzia),
\quad
s\longleftrightarrow \pi_s.
\end{equation}
The order of the connected components of $s$ is the order of the
corresponding blocks by increasing minima, because the root of each
component is its minimal vertex in the root-subtree order. Fix $s\in\mathrm{Cov}(t)$ and let
\[
\pi_s=(\beta_1,\ldots,\beta_m).
\]
Let $u_\ell$ be the connected component of $s$ corresponding to
\[
\beta_\ell=\{x_{j_1}<\cdots<x_{j_q}\}.
\]
For each $1\le r\le q$, the number
$\operatorname{out}_{\pi_s}(x_{j_r})$ is the number of children of
$v_{j_r}$ in $t$ which do not belong to $u_\ell$.  Hence the fertility of
$v_{j_r}$ inside $u_\ell$ is
\[
{\rm wt}(x_{j_r})+1-\operatorname{out}_{\pi_s}(x_{j_r}),
\]
and its decoration remains $d(x_{j_r})$.  Therefore
\[
\Phi_{\mathrm{fer}}^{\mathcal T}(u_\ell)
=
\beta_{\ell,\pi_s}^{\circ}.
\]
Reading the components in their planar order gives
\begin{equation}\label{eq:ec-forestwordA}
\Phi_{\mathrm{fer}}^{\mathcal F}(s)=M_0(\pi_s).
\end{equation}

It remains to compare the contracted terms.  The vertices of $t/s$ are the
connected components $u_1,\ldots,u_m$, ordered by their roots.  The vertex
corresponding to $u_\ell$ is decorated by
$d(\operatorname{rt}(u_\ell))=\rho(\beta_\ell).$
Moreover, the number of children of this contracted vertex is
$\sum_{x_j\in\beta_\ell}\operatorname{out}_{\pi_s}(x_j)={\rm wt}(\beta_\ell)+1$
by Lemma~\ref{lem:ec-blockA}.  Hence the corresponding letter in the
fertility word of $t/s$ is
$z^{\rho(\beta_\ell)}_{{\rm wt}(\beta_\ell)}.$
Therefore
\begin{equation}\label{eq:ec-quotientwordA}
\Phi_{\mathrm{fer}}^{\mathcal T}(t/s)=\bzia/\pi_s.
\end{equation}

\noindent Using the bijection \eqref{eq:ec-bijectionA}, together with
\eqref{eq:ec-forestwordA} and \eqref{eq:ec-quotientwordA}, we obtain
\[
\begin{aligned}
(J_{\mathcal F}\otimes J_{\mathcal F})
\Gamma_{\mathrm{PLOT}}(\bzia)
&=
\sum_{\pi\in\mathrm{Cov}(\bzia)}
J_{\mathcal F}(M_0(\pi))\otimes J(\bzia/\pi)\\
&=
\sum_{s\in\mathrm{Cov}(t)}
s\otimes t/s\\
&=
\Gamma_{\mathrm{MKW}}(t)
=
\Gamma_{\mathrm{MKW}}J_{\mathcal F}(\bzia).
\end{aligned}
\]
This proves the identity on generators, and the conclusion follows by
multiplicativity.
\end{proof}

\begin{remark}
The crucial point in the planar/noncommutative setting is the canonical choice
\eqref{eq:ec-rootdecA} of the decoration of a contracted block: a connected block has a unique
root, namely its minimal vertex in the root-subtree order. As in the planar cut formula,
no symmetry factor appears because the planar order rigidifies both the blocks and their order.
\mlabel{rem:ec-generalA}
\end{remark}

\section{Symmetrization operator}\label{sec:symmetrization}
The purpose of this section is to construct the multi-index version of the symmetrization operator satisfying the commutative diagram in \eqref{eq:symmetrization-square}.
Let us begin with the recalling of necessary notations in the commutative framework~\cite{ZHU2024}. Let
\[
{\rm N}(A):=\bfk[\bar{x}_j^a\mid (a,j)\in A\times \mathbb Z_{\ge -1}]
\]
be the commutative polynomial algebra, whose commutative concatenation product is denoted by $\odot$.
Here, we add a ``bar" above the variables to distinguish it from the noncommutative case.
Denote by ${\rm N}(A)_{-1}$ the homogeneous component of weight $-1$.
We write $\mathcal M(A)$ for the set of commutative monomial forests, namely
\[
\mathcal M(A):=
\{\etree\}
\cup
\Bigl\{
 M_1^{\odot \ell_1}\odot \cdots \odot
 M_r^{\odot \ell_r}
\ \Big|\
r\ge 1,\ 
 M_i\in {\rm N}(A)_{-1}\ \text{are pairwise distinct},\
\ell_i\ge 1
\Bigr\}.
\]
Thus $\mathcal M(A)$ forms the $\bfk$-basis of the LOT Hopf algebra
$\mathcal H_{\mathrm{LOT}}^A$~\cite{ZHU2024}.\\

\noindent For a weight $-1$ commutative monomial
\[
M:=\bx^{\overline{\bfk}}
:=
\prod_{a\in A,\ j\ge -1}(\bar{x}_{j}^{a})^{k_j^a}
\in {\rm N}(A)_{-1},
\quad
k_j^a\in \ZZ_{\geq 0},
\]
 define its {\bf symmetry factor} by
\[
\sigma(M)
:=
\sigma(\bx^{\overline{\bfk}})
:=
\prod_{a\in A,\ j\ge -1} k_j^a!,
\]
where $\overline{\bfk}$ stands for the multi-index $(k_j^a)_{a\in A,\, j\ge -1}$. For
\[
\mathfrak M
=
M_1^{\odot \ell_1}\odot \cdots \odot
M_r^{\odot \ell_r}
\in \mathcal M(A),
\]
with $ M_1,\ldots,M_r\in{\rm N}(A)_{-1}$ pairwise
distinct, define
\[
\sigma_{\mathrm{w}}(\mathfrak M)
:=
\ell_1!\cdots \ell_r!\,
\sigma(M_1)^{\ell_1}\cdots
\sigma(M_r)^{\ell_r}.
\]

We next recall the rooted tree side symmetrization operator~\cite{MKW2008}.
Let $\mathcal T(A)$ and $\mathcal F(A)$ denote the sets of $A$-decorated rooted trees and rooted
forests without planar structure, and let
\[
\FF:\OT^A \longrightarrow\mathcal T(A),
\qquad
\FF:\OF^A \longrightarrow \mathcal F(A)
\]
be the {\bf forgetful maps}.
The tree {\bf symmetrization operator} is given by
\begin{equation}
\Omega_t:\bfk\mathcal F(A)\longrightarrow \bfk\OF^A,
\quad
f\longmapsto
\sigma_{\tau}(f)
\sum_{\widetilde f\in \OF^A,\ 
\mathrm{\FF}(\widetilde f)=f}
\widetilde f .
\mlabel{eq:Omegatexp}
\end{equation}
It is an injective Hopf algebra morphism:
\begin{equation*}
\mathcal{H}_{\rm BCK}^{A}\inj \mathcal{H}_{\rm MKW}^{A}.
\end{equation*}
 Here $\sigma_{\tau}(f)$ is the classical symmetry factor of the rooted forest $f$.
We also recall the commutative version of the fertility map on non-planar rooted trees~\cite{ZHU2024}:
\begin{equation*}
\Phi_{\mathrm{com}}^{\mathcal T}:\bfk\mathcal T(A)\longrightarrow {\rm N}(A)_{-1},
\quad
\tau\longmapsto \prod_{v\in V(\tau)} \bar{x}_{f(v)-1}^{d(v)},
\mlabel{eq:comm-fer-tree}
\end{equation*}
where $f(v)$ is the fertility of the vertex $v$. The embedding constructed in~\cite{ZHU2024} is given by
\begin{equation}
j:\mathcal{H}_{\mathrm{LOT}}^A \hookrightarrow \mathcal{H}_{\mathrm{BCK}}^A,\quad \bx^{\overline{\bfk}} \mapsto \sum_{\tau\in \mathcal T(A),\ \Phi_{\mathrm{com}}^{\mathcal T}(\tau)=\bx^{\overline{\bfk}}}
\frac{\sigma_{\mathrm{w}}(\bx^{\overline{\bfk}})}{\sigma_{\mathrm{\tau}}(\tau)}\tau,\quad\text{ for } \bx^{\overline{\bfk}}\in {\rm N}(A)_{-1}.
\mlabel{eq:j-generator}
\end{equation}

\noindent We now construct the corresponding symmetrization operator on the multi-index side.

\begin{definition}
Define the {\bf abelianization operator}
\begin{equation*}
\operatorname{Ab}: V(A) \longrightarrow {\rm N}(A)_{-1}, \quad  \bzia=x_1\cdots x_p\in V(A)\mapsto \prod_{m=1}^p \bar{x}_{\wt(x_m)}^{d(x_m)},
\mlabel{eq:ab-word}
\end{equation*}
and extend multiplicatively to
\begin{equation*}
\operatorname{Ab}:  \mathcal{W}(A)\rightarrow \mathcal M(A),
\mlabel{eq:Ab-word-forest}
\end{equation*}
where $\mathcal{W}(A)$ is given in \meqref{eq:zzzz}.
\mlabel{def:ab-word}
\end{definition}

\begin{lemma}\label{lem:ab-fertility}
For every $t\in \OT^A$, one has
\begin{equation*}
\operatorname{Ab}\bigl(\Phi_{\mathrm{fer}}^{\mathcal T}(t)\bigr)
=
\Phi_{\mathrm{com}}^{\mathcal T}\bigl(\FF(t)\bigr).
\end{equation*}
\end{lemma}

\begin{proof}
Both sides are the commutative product, over all vertices $v\in V(t)$, of the variable indexed by
the decoration $d(v)$ and the fertility $f(v)-1$.
The planar order only affects the order of the letters in $\Phi_{\mathrm{fer}}^{\mathcal T}(t)$, and this
order disappears after abelianization.
\end{proof}

\begin{definition}
Define the {\bf word  symmetrization operator}
\begin{equation}
\Omega_w:\mathcal{H}_{\mathrm{LOT}}^A\longrightarrow \mathcal{H}_{\mathrm{PLOT}}^A
\mlabel{omegaw}
\end{equation}
by setting
\begin{equation}
\left\{
\begin{aligned}
\Omega_w(\etree)
&:=\etree,\\
\Omega_w(M)
&:=
\sigma_{\mathrm{w}}(\bx^{\overline{\bfk}})
\sum_{\bzia\in V(A),\ \operatorname{Ab}(\bzia)=\bx^{\overline{\bfk}}}\bzia,
\quad \forall M=\bx^{\overline{\bfk}}\in {\rm N}(A)_{-1},\\
\Omega_w(\mathfrak M \odot \mathfrak N)
&:=\Omega_w(\mathfrak M)\shuffle \Omega_w(\mathfrak N),
\quad \forall \mathfrak M ,\mathfrak N\in \mathcal M(A).
\end{aligned}
\right.
\mlabel{eq:Omega-w-def}
\end{equation}
\mlabel{def:Omega-w}
\end{definition}

\begin{proposition}\label{prop:Omega-w-explicit}
With the above setting,
\begin{enumerate}
\item the map $\Omega_w$ is a well-defined injective algebra morphism. \mlabel{it:symma}

\item for every basis element
\[
\mathfrak M=M_1^{\odot \ell_1}\odot \cdots \odot M_r^{\odot \ell_r}\in \mathcal M(A),
\]
with $M_1,\ldots,M_r\in {\rm N}(A)_{-1}$ pairwise distinct, one has
\begin{equation}\label{eq:Omega-w-explicit}
\Omega_w(\mathfrak M)=
\sigma_{\mathrm{w}}(\mathfrak M)
\sum_{W\in {\mathcal W}(A),\ \operatorname{Ab}(W)=\mathfrak M} W.
\end{equation}
\mlabel{it:symmb}
\end{enumerate}
\end{proposition}

\begin{proof}
(a) Since the product $\odot$ on $\mathcal H_{\mathrm{LOT}}^A$ is commutative and the shuffle product $\shuffle$ on
$\mathcal H_{\mathrm{PLOT}}^A$ is also commutative, the rule in Definition~\ref{def:Omega-w} does not depend on the
order in which the factors of $\mathfrak M$ are written.
Hence $\Omega_w$ is well defined.

(b) Let
$
\mathfrak M=M_1^{\odot \ell_1}\odot \cdots \odot M_r^{\odot \ell_r}.
$
By repeated use of Definition~\ref{def:Omega-w}, we obtain
\[
\Omega_w(\mathfrak M)
=
\Omega_w(M_1)^{\shuffle \ell_1}\shuffle \cdots \shuffle \Omega_w(M_r)^{\shuffle \ell_r}.
\]
Now fix an ordered monomial forest
\[
W=\bw_1\cdots \bw_n\in {\mathcal W}(A)\, \text{ such that }\,
\operatorname{Ab}(W)=\mathfrak M.
\]
For each $i\in\{1,\dots,r\}$, the $\ell_i$ factors whose abelianization is $M_i$ can be assigned to their
positions in $\ell_i!$ ways.
Therefore $W$ appears exactly
$\ell_1!\cdots \ell_r!$
times in the above expansion.
Its coefficient is thus
\[
\ell_1!\cdots \ell_r!\,
\sigma_{\mathrm{w}}(M_1)^{\ell_1}\cdots \sigma_{\mathrm{w}}(M_r)^{\ell_r}
=
\sigma_{\mathrm{w}}(\mathfrak M),
\]
which proves~\eqref{eq:Omega-w-explicit}.
Finally, injectivity follows from~\eqref{eq:Omega-w-explicit}.
Indeed, for any $W\in {\mathcal W}(A)$, the value of $\operatorname{Ab}(W)$ is uniquely determined by $W$.
Hence two distinct basis elements of $\mathcal M(A)$ cannot have a common ordered monomial lift in
${\mathcal W}(A)$, and so their images under $\Omega_w$ have disjoint supports.
Therefore $\Omega_w$ is injective.
\end{proof}

\noindent Let us now state the main result in this section.

\begin{theorem}
The maps $j$, $\Omega_t$, $\Omega_w$, and $J$ fit into the commutative square
\begin{equation}
\xymatrix{
 \mathcal{H}_{{\rm LOT}}^{A} \ar@{^{(}->}[d]_{\Omega_w}\ar@{^{(}->}[r]^{j}
 & \mathcal{H}_{{\rm BCK}}^{A}  \ar@{^{(}->}[d]^{\Omega_t}\\
 \mathcal{H}_{{\rm PLOT}}^{A}  \ar[r]^{J}_{\sim}
 & \mathcal{H}_{{\rm MKW}}^{A}
}
\mlabel{eq:symmetrization-square}
\end{equation}
\mlabel{thm:symmetrization-square}
\end{theorem}

\begin{proof}
Let $\bx^{\overline{\bfk}}\in {\rm N}(A)_{-1}$. On the one hand,
\begin{align}
\Omega_t\circ\bigl(j(\bx^{\overline{\bfk}})\bigr)
&\overset{(\ref{eq:j-generator})}{=}
\sum_{\tau\in \mathcal T(A),\ \Phi_{\mathrm{com}}^{\mathcal T}(\tau)=\bx^{\overline{\bfk}}}
\frac{\sigma_{\mathrm{w}}(\bx^{\overline{\bfk}})}{\sigma_{\mathrm{\tau}}(\tau)}\,
\Omega_t(\tau)  \nonumber
\\
&\overset{\eqref{eq:Omegatexp}}{=}
\sum_{\tau\in \mathcal T(A),\ \Phi_{\mathrm{com}}^{\mathcal T}(\tau)=\bx^{\overline{\bfk}}}
\frac{\sigma_{\mathrm{w}}(\bx^{\overline{\bfk}})}{\sigma_{\tau}(\tau)}\,
\sigma_{\tau}(\tau)
\sum_{t\in \OT^A,\ \FF(t)=\tau}t  \nonumber
\\
&=
\sigma_{\mathrm{w}}(\bx^{\overline{\bfk}})
\sum_{t\in \OT^A,\ \Phi_{\mathrm{com}}^{\mathcal T}(\FF(t))=\bx^{\overline{\bfk}}}t \nonumber
\\
&= \sigma_{\mathrm{w}}(\bx^{\overline{\bfk}})
\sum_{t\in \OT^A,\ \operatorname{Ab}\bigl(\Phi_{\mathrm{fer}}^{\mathcal T}(t)\bigr)=\bx^{\overline{\bfk}}}t \hspace{1cm}\text{(by Lemma~\ref{lem:ab-fertility})}.
\mlabel{eq:left}
\end{align}
On the other hand,
\begin{align}
J\circ\bigl(\Omega_w(\bx^{\overline{\bfk}})\bigr)
& \overset{(\ref{eq:Omega-w-def})}{=}
\sigma_{\mathrm{w}}(\bx^{\overline{\bfk}})
\sum_{\bzia\in V(A),\ \operatorname{Ab}(\bzia)=\bx^{\overline{\bfk}}} J(\bzia)\nonumber
\\
& =
\sigma_{\mathrm{w}}(\bx^{\overline{\bfk}})
\sum_{t\in \OT^A,\ \operatorname{Ab}\bigl(\Phi_{\mathrm{fer}}^{\mathcal T}(t)\bigr)=\bx^{\overline{\bfk}}}t  \hspace{2cm} \text{(by $J=(\Phi_{\mathrm{fer}}^{\mathcal T})^{-1}$)}.
\mlabel{eq:right}
\end{align}
Comparing~\eqref{eq:left} and~\eqref{eq:right}, we obtain the required equality. Since $j$, $\Omega_t$, $\Omega_w$, and $J$ are algebra morphisms\footnotemark, and since
$\mathcal H_{\mathrm{LOT}}^A$ is generated as an algebra by ${\rm N}(A)_{-1}$,
the identity extends from generators to all elements of
$\mathcal H_{\mathrm{LOT}}^A$.  This completes the proof.
\end{proof}

\footnotetext{This follows from the construction of $j$ recalled before~\eqref{eq:j-generator},
\eqref{eq:Omegatexp}, Proposition~\mref{prop:Omega-w-explicit}, and
\eqref{J:hopfm}, respectively.}

\begin{coro}\label{cor:Omega-w-hopf}
The word symmetrization operator $\Omega_w$ in Definition~\mref{def:Omega-w}
is an injective Hopf algebra morphism.
\end{coro}

\begin{proof}
By Proposition~\ref{prop:Omega-w-explicit}, it remains to prove that $\Omega_w$ is compatible with the coproducts.
 We obtain
\begin{align*}
(J\otimes J)\circ \Delta_{\mathrm{PLOT}}\circ \Omega_w
&\overset{\eqref{J:hopfm}}{=}
\Delta_{\mathrm{MKW}}\circ J\circ \Omega_w
\\
&\overset{\eqref{eq:symmetrization-square}}{=}
\Delta_{\mathrm{MKW}}\circ \Omega_t \circ j
\\
&=
(\Omega_t\otimes \Omega_t)\circ \Delta_{\mathrm{BCK}}\circ j \hspace{1cm} \text{(by $\Omega_t$ being a Hopf algebra morphism)}
\\
&=
(\Omega_t\otimes \Omega_t)\circ (j\otimes j)\circ \Delta_{\mathrm{LOT}} \hspace{1cm} \text{(by $j$ being a Hopf algebra morphism)}
\\
&\overset{\eqref{eq:symmetrization-square}}{=}
(J\otimes J)\circ (\Omega_w\otimes \Omega_w)\circ \Delta_{\mathrm{LOT}}.
\end{align*}
Since $J$ is an isomorphism, so is $J\otimes J$.
Hence
\[
\Delta_{\mathrm{PLOT}}\circ\Omega_w
=
(\Omega_w\otimes \Omega_w)\circ\Delta_{\mathrm{LOT}}.
\]
Thus $\Omega_w$ is a bialgebra morphism.
As both Hopf algebras are connected and graded, $\Omega_w$ is automatically a Hopf algebra morphism.
\end{proof}

\noindent
{\bf Acknowledgments.} This work is supported by the National Natural Science Foundation of China (12571019), the Natural Science Foundation of Gansu Province (25JRRA644), Innovative Fundamental Research Group Project of Gansu Province (23JRRA684) and Longyuan Young Talents of Gansu Province.

\end{document}